\documentclass[12pt]{article}
\usepackage[utf8]{inputenc}
\usepackage[margin=1in]{geometry}
\usepackage{amsmath}
\usepackage{algorithm}
\usepackage[noend]{algpseudocode}
\usepackage{amsfonts}
\usepackage{amssymb}
\usepackage{bm}
\usepackage[numbers]{natbib}
\usepackage{bbm}
\usepackage{textcomp}
\usepackage{pgfplots}
\usepackage{bbm}
\usepackage{graphicx}
\usepackage{wrapfig}
\usepackage[]{youngtab}
\usepackage{amsthm}
\usepackage{tikz}
\usepackage{tikz-cd}
\usetikzlibrary{arrows.meta,automata, positioning}
\usepackage{comment}
\usepackage{stmaryrd}

\usepackage{theoremref}
\usepackage{mathrsfs}
\usepackage{color-edits}
\addauthor{sw}{blue}
\addauthor{rr}{red}
\addauthor{jw}{brown}
\addauthor{ar}{red}

\usepackage[english]{babel}
\usepackage[utf8]{inputenc}
\usepackage{fancyhdr}
\pgfplotsset{width=10cm,compat=1.9}

\renewcommand{\P}{\mathbb{P}}
\newcommand{\N}{\mathbb{N}_{\mathbf{0}}}
\newcommand{\E}{\mathbb{E}}
\newcommand{\R}{\mathbb{R}}
\renewcommand{\S}{\mathbb{S}}

\newcommand{\calF}{\mathcal{F}}
\newcommand{\calX}{\mathcal{X}}

\newcommand{\calG}{\mathcal{G}}
\newcommand{\calL}{\mathcal{L}}

\newcommand{\calC}{\mathcal{C}}

\newcommand{\calN}{\mathcal{N}}
\newcommand{\calS}{\mathcal{S}}

\newcommand{\fraks}{\mathfrak{s}}

\newcommand{\id}{\mathrm{id}}

\newcommand{\vX}{\mathbf{X}}
\newcommand{\vY}{\mathbf{Y}}
\newcommand{\wh}[1]{\widehat{#1}}
\newcommand{\wt}[1]{\widetilde{#1}}
\newcommand{\wb}[1]{\overline{#1}}
\newcommand{\AR}[1]{\mathrm{AR}(#1)}
\newcommand{\VAR}[1]{\mathrm{VAR}(#1)}
\newcommand{\C}{\mathbb{C}}
\newcommand{\B}{\mathbb{B}}

\newcommand\numberthis{\addtocounter{equation}{1}\tag{\theequation}}

\newtheorem{theorem}{Theorem}
\numberwithin{theorem}{section}
\newtheorem{lemma}[theorem]{Lemma}

\newtheorem{prop}[theorem]{Proposition}
\newtheorem{corollary}[theorem]{Corollary}
\newtheorem{fact}[theorem]{Fact}
\newtheorem{model}[theorem]{Model}

\theoremstyle{definition}
\newtheorem{definition}[theorem]{Definition}
\newtheorem{example}[theorem]{Example}

\theoremstyle{remark}

\usepackage{thm-restate}
\usepackage{hyperref}
\hypersetup{
	colorlinks=false,
	bookmarks=true,
	breaklinks=true,
	hidelinks=true,
	pdfpagemode=empty,
}
\usepackage{nameref}
\usepackage{subfig}
\usepackage{authblk}
\usepackage[noabbrev, capitalize, nameinlink]{cleveref}
\crefname{prop}{Proposition}{Propositions}
\crefname{rmk}{Remark}{Remarks}
\crefname{cor}{Corollary}{Corollaries}
\crefname{claim}{Claim}{Claims}
\crefname{lemma}{Lemma}{Lemmata}
\crefname{example}{Example}{Examples}
\crefname{corollary}{Corollary}{Corollaries}
\title{Time-Uniform Self-Normalized Concentration\\ for Vector-Valued Processes}
\author[1]{Justin Whitehouse}
\author[2]{Zhiwei Steven Wu} \author[2]{Aaditya Ramdas}
\affil[1]{Stanford University }
\affil[2]{Carnegie Mellon University}
\affil[ ]{\texttt{ jwhiteho@stanford.edu, 
\{zstevenwu,aramdas\}@andrew.cmu.edu}}
\date{\today}

\begin{document}
\maketitle
\begin{abstract}
Self-normalized processes arise naturally in many learning-related tasks. While self-normalized concentration has been extensively studied for scalar-valued processes, there are few results for multidimensional processes  outside of the sub-Gaussian setting. In this work, we construct a general, self-normalized inequality for $\R^d$-valued processes that satisfy a simple yet broad ``sub-$\psi$'' tail condition, which generalizes assumptions based on cumulant generating functions.   
From this general inequality, we derive an upper law of the iterated logarithm for sub-$\psi$ vector-valued processes, which is tight up to small constants. 
We show how our inequality can be leveraged to derive a variety of novel, self-normalized concentration inequalities under both light and heavy-tailed observations.
Further, we provide applications in prototypical statistical tasks, such as parameter estimation in online linear regression, autoregressive modeling, and bounded mean estimation via a new (multivariate) empirical Bernstein concentration inequality.
\end{abstract}
\section{Introduction}
Concentration inequalities are employed in many disparate mathematical fields. In particular, time-uniform martingale concentration has proven itself a critical tool in advancing research areas such as multi-armed bandits~\citep{kaufmann2016complexity, abbasi2011improved, lattimore2020bandit}, differential privacy~\citep{whitehouse2022fully, whitehouse2022brownian}, Bayesian learning~\citep{chugg2023unified}, and online convex optimization~\citep{li2020high,jun2019parameter}. While martingale concentration inequalities have historically been proved in a largely case-by-case manner, recently \citet{howard2020time} provided a unified framework for constructing time-uniform concentration inequalities. By introducing a single ``sub-$\psi$'' assumption that carefully controls the tail behavior of martingale increments, \citet{howard2020time, howard2021time} prove a master theorem that recovers (in fact improves) many classical examples of concentration inequalities, for example those of~\citet{blackwell1997large,hoeffding1963probability, freedman1975tail, azuma1967weighted, de2004self}.

Despite the generality of the framework of \citet{howard2020time, howard2021time}, their results have not been extended to understanding the growth of ``self-normalized'' vector-valued processes. 
If $(S_t)_{t \geq 0}$ is a process evolving in $\R^d$ and $(V_t)_{t \geq 0}$ is a process of $d \times d$ positive semi-definite matrices measuring the ``accumulated variance'' of $(S_t)_{t \geq 0}$, self-normalized concentration aims to control the growth of the normalized process $(\|V_t^{-1/2}S_t\|)_{t \geq 0}$. Self-normalized processes naturally arise in a variety of common statistical tasks, examples of which include regression problems~\citep{lai1981consistency, lai1982least, bercu2008exponential} and contextual bandit problems~\citep{abbasi2011improved, chowdhury2017kernelized}. As such, any advances in self-normalized concentration for vector-valued processes could directly yield improvements in methodology and analysis of foundational statistical algorithms.

In this work, we provide a new, general approach for constructing self-normalized concentration inequalities. By naturally generalizing the sub-$\psi$ condition of \citet{howard2020time, howard2021time} to $d$-dimensional spaces, we are able to construct a single ``master'' theorem that provides time-uniform, self-normalized concentration under a variety of noise settings. We prove our results by first constructing a time-uniform concentration inequality for scalar-valued processes that non-asymptotically matches law of the iterated logarithm and then extending this result to higher dimensions using a geometric argument. 
From our inequality, we can derive a multivariate analogue of the famed law of the iterated logarithm, which we show to be essentially tight. Lastly, we apply our inequality to common statistical tasks, such as calibrating confidence ellipsoids in online linear regression, estimating model parameters in vector auto-regressive models, and estimating a bounded mean via a new ``empirical Bernstein'' concentration inequality.

\subsection{Related Work and History}

Martingale concentration arguably originated in the work of \citet{ville1939etude}, who showed that the growth of  non-negative  supermartingales can be controlled uniformly over time. This result, now known commonly referred to as \textit{Ville's inequality}, acts as a time-uniform generalization of Markov's inequality~\citep{durrett2019probability}. This result was later extended to submartingale concentration by \citet{doob1940regularity} in an eponymous result, now called \textit{Doob's maximal inequality}. From these two inequalities, a variety of now classical martingale concentration inequalities were proved, such as Azuma's inequality~\citep{azuma1967weighted}, which serves as a time-uniform, martingale variant of Hoeffding's inequality~\citep{hoeffding1963probability} for bounded random variables, and Freedman's inequality~\citep{freedman1975tail}, which serves as a martingale variant of Bennett's inequality~\citep{bennett1962probability} for sub-Poisson, bounded random variables. 

Of particular note are the various self-normalized inequalities of de la Pe\~na~\citep{victor2009theory,de2004self, de2007pseudo, pena2009self}, which provide time-uniform control of the growth of a process $(S_t)_{t \geq 0}$ in terms of an associated accumulated variance process $(V_t)_{t \geq 0}$. In particular, the authors derive their results using a technique first presented by Robbins called the method of mixtures~\citep{darling1968some, darling1967iterated}, which involves integrating over a family of parameterized exponential supermartingales to obtain significantly tighter (in terms of asymptotic behavior) inequalities than those mentioned earlier. \citet{bercu2008exponential} also investigate self-normalized concentration in the style of de la Pe{\~n}a, deriving bounds when the increments of $(S_t)_{t \geq 0}$ may exhibit asymmetric heavy-tailed behavior and, in later work, \citep{bercu2019new} study the effects of weighing predictable and empirical quadratic variation processes in deriving self-normalized concentration results. 

Recently, \citet{howard2020time} presented a single ``master'' theorem that ties together much of the literature surrounding scalar-valued  concentration (self-normalized or not). Inspired by the classical Cramer-Chernoff method (see \citet{boucheron2013concentration} for instance), which provides high probability tail bounds for a random variable $X$ in terms of its cumulant generating function (or CGF) $\psi$, the authors present a unified ``sub-$\psi$'' condition on a stochastic process. This condition relates the growth of a process $(S_t)_{t \geq 0}$ to some corresponding accumulated variance process $(V_t)_{t \geq 0}$ through a function $\psi$ which obeys many similar properties to a CGF. In particular, the authors prove ``line-crossing'' inequalities for sub-$\psi$ processes, giving a bound on the probability that $(S_t)_{t \geq 0}$ will ever cross a line parameterized by $\psi$ and the accumulated variance $(V_t)_{t \geq 0}$. By strategically picking $\psi$ and $(V_t)_{t \geq 0}$, the master theorem in \citet{howard2020time} can be used to reconstruct, unify and even improve a variety of existing self-normalized concentration inequalities (such as those in the preceding paragraph), as well as to prove several new ones. Using these ideas in a followup work, \citet{howard2021time} prove a time-uniform concentration inequality for scalar-valued processes whose rate non-asymptotically matches the law of the iterated logarithm (LIL)~\citep{durrett2019probability}. The only caveat to this result is that the concentration inequality only applies to sub-$\psi$ processes when $\psi$ is either the CGF of a sub-Gaussian (denoted $\psi_N$) or sub-Gamma (denoted $\psi_{G, c}$) random variable. While any CGF-like function $\psi$ function can be bounded by $a \psi_{G, c}$ for \textit{some} choice of $a, c > 0$ (see Proposition 1 of \citet{howard2021time}), this conversion could in general result in loose constants.
As a stepping stone toward proving our multivariate concentration inequalities, we generalize the non-asymptotic LIL results of \citet{howard2021time} to arbitrary sub-$\psi$ process, greatly increasing the applicability of the obtained results.

To the best of our knowledge, there are relatively few existing results on the self-normalized concentration of vector-valued processes. De la Pe\~na \citep{victor2009theory} leverage the above-mentioned method of mixtures alongside Ville's inequality to bound the probability that the self-normalized random vector $V_{t}^{-1/2}S_t$  belongs to some mixture-dependent convex set $\Gamma_t \subset \R^d$. These bounds are, in particular, not closed form, and it is also unclear how to translate said bounds into finite sample bounds on $\|V_t^{-1/2}S_t\|$. 
Our bounds, instead, directly provide time-uniform bounds on the process $(\|V_t^{-1/2}S_t\|)_{t \geq 0}$ in terms of relatively simple function of the variance process $(V_t)_{t \geq 0}$. 
In the same paper, the aforementioned authors 
also prove an asymptotic law of the iterated logarithm for self-normalized, vector-valued processes, which is stated in terms of the maximum eigenvalue and condition number of $V_t$. Unfortunately, these bounds are weak in that they only capture asymptotic rate of growth up to a \textit{random} constant. Using our bounds, we are able to derive a multivariate law of the iterated logarithm that is tight in terms of dependence on $V_t$ and the ambient dimension $d$ up to small, absolute, known constants. 

In the case where the increments of $(S_t)_{t \geq 0}$ satisfy a sub-Gaussian condition, significantly more is known about vector-valued self-normalized concentration. \citet{abbasi2011improved} provide a clean bound on $\|V_t^{-1/2}S_t\|$ in terms of $\log\det(V_t)$ using an argument that directly follows from an earlier, method-of-mixtures based argument of \citet{de2007pseudo}. First, our bounds are significantly more general than those of \citet{abbasi2011improved} and \citet{de2007pseudo}, because ours apply to \emph{general} sub-$\psi$ processes. Additionally, our bounds grow proportionally to $\log\log\gamma_{\max}(V_t)$ and $\log\kappa(V_t)$ ($\gamma_{\max}$ and $\kappa$ represent maximum eigenvalue and condition number respectively, defined later). Thus, even in the setting of sub-Gaussian increments with predictable covariance, our results are not directly comparable in general. We believe deriving log-determinant rate inequalities for general sub-$\psi$ processes is an interesting open problem, but leave it for future work.

There exist other  concentration inequalities for vector-valued data that are not directly related to the self-normalized bounds presented in this paper. First, there are several existing time-uniform concentration results for Banach space-valued martingales~\citep{pinelis1992approach, pinelis1994optimum, howard2020time}. These results are obtained by placing a smoothness assumption on the norm of the Banach space, and in turn provide time-uniform control on the norm of the martingale. We note that although we are working in a Banach space, we are not trying to control the norm of the underlying process $\|S_t\|$, and instead want to control the self-normalized quantity $\left\|V_t^{-1/2}S_t\right\|$. In particular, the process $(V_t^{-1/2}S_t)_{t \geq 0}$ is not in general a martingale, so the above results cannot be directly applied. Second, there are many concentration results that involve bounding the operator norm of Hermitian matrix-valued martingales using the matrix Chernoff method~\citep{ahlswede2002strong,christofides2008expansion,tropp2011freedman,tropp2012user}. Once again, it does not seem like these bounds for matrix-valued processes can be readily applied to obtain vector-valued concentration of the form presented in this paper. Third, in their work on estimating convex divergences, \citet{manole2021sequential} derive a self-normalized concentration inequality for i.i.d. random vectors drawn from some distribution on $\R^d$. The form of this bound resembles that of the central concentration inequality presented in this paper. However, we note that our result allows for arbitrary martingale dependence between the increments of the process $(S_t)_{t \geq 0}$. Furthermore, the argument used in \citet{manole2021sequential} cannot be generalized to the setting of arbitrary dependence, as the authors derive their results using certain reverse martingale arguments which must be conducted with respect to the exchangeable filtration generated by a sequence of random variables, which implies the increments of $(S_t)_{t \geq 0}$ must, at the very least, be exchangeable random variables. 

\subsection{Our Contributions}

We now provide a brief, illustrative summary of our primary contributions. For now, when we refer to a process $(S_t)$ being sub-$\psi$ with variance proxy $(V_t)$, the reader should think of the increments of $S_t$ having associated cumulant generating function (CGF) $\psi$ with weights proportional to $V_t$. This is not precise, but will be made exact when we provide rigorous definitions of the sub-$\psi$ condition for both scalar and vector-valued processes in Section~\ref{sec:background} below. Likewise, we denote by $\psi^\ast$ the convex conjugate of a given function $\psi$. We present the primary contributions in the order they appear in the paper.
\begin{enumerate}
    \item First, in Theorem~\ref{thm:scalar} of Section~\ref{sec:scalar}, we show that if $(S_t)_{t \geq 0}$ is a scalar (i.e.\ $\R$-valued) sub-$\psi$ process with variance proxy $(V_t)_{t \geq 0}$, then, with high probability, it holds that
    \[
    S_t  \lesssim V_t \cdot (\psi^\ast)^{-1}\left(\frac{1}{V_t}\log\log(V_t)\right)
    \]
    for all $t \geq 0$ simultaneously. In the case where $\psi(\lambda) = \psi_{G, c}(\lambda) := \frac{\lambda^2}{2(1 - c\lambda)}$ is the CGF associated with a sub-Gamma random variable (see \citet{boucheron2013concentration}), our bound reduces to 
    \[
    S_t  \lesssim \sqrt{V_t\log\log(V_t)} + c\log\log(V_t).
    \]
    Thus, this result can be reviewed as a direct generalization of the primary contributions of \citet{howard2021time}, who only provide time-uniform, self-normalized concentration results for sub-Gamma processes (note that in the special case $c = 0$, sub-Gamma concentration reduces to sub-Gaussian concentration). In Corollary~\ref{cor:lil_scalar}, we use our bound to prove a law of the iterated logarithm for sub-$\psi$ processes with the correct constant, showing asymptotically that our bounds are unimprovable. 
    \item Next, in Theorem~\ref{thm:vector} of Section~\ref{sec:vector}, we show that if $(S_t)_{t \geq 0}$ is a vector valued process that is sub-$\psi$ with variance proxy $(V_t)_{t \geq 0}$, then, with high probability, simultaneously for all $t \geq 0$,
    \[
    \left\|V_t^{-1/2}S_t\right\| \lesssim \sqrt{\gamma_{\min}(V_t)}\cdot (\psi^\ast)^{-1}\left(\frac{1}{\gamma_{\min}(V_t)}\left[\log\log(\gamma_{\max}(V_t)) + d\log \kappa(V_t)\right]\right), 
    \]
    where $\gamma_{\min}(V_t)$, $\gamma_{\max}(V_t)$, and $\kappa(V_t)$ are, respectively, to the minimum eigenvalue, maximum eigenvalue, and condition number of the matrix $V_t$. Once again, for sub-$\psi_{G, c}$ processes, our bound becomes
    \[
    \left\|V_t^{-1/2}S_t\right\| \lesssim \sqrt{\log\log\gamma_{\max}(V_t) + d\log\kappa(V_t)} + c\frac{\log\log\gamma_{\max}(V_t) + d \log \kappa(V_t)}{\sqrt{\gamma_{\min}(V_t)}}
    \]
    We $c = 0$, we can compare our result to existing sub-Gaussian bounds~\citep{de2007pseudo, abbasi2011improved}. Existing rates in this setting are of the form  $\|V_t^{-1/2}S_t\| = O\left(\sqrt{\log\det(V_t)}\right)$, which are, in general, incomparable to our bounds. When $\kappa(V_t)$ is small, our bounds may be tighter, but if $\gamma_{\max}(V_t) \gg \gamma_{\min}(V_t)$, the determinant-based bounds may be tighter. We additionally prove a multivariate law of the iterated logarithm for self-normalized processes in Corollary~\ref{cor:lil_vec} and further provide a counterexample to show that (up to small constants) the rate we achieve is unimprovable.

    \item Lastly, in Section~\ref{sec:apps}, we apply our vector-valued self-normalized concentration results to statistical tasks. First, in Subsection~\ref{subsec:reg}, we create non-asymptotically valid confidence ellipsoids for estimating  unknown slope parameters in online linear regression with sub-$\psi$ noise in observations. In particular, these results can be viewed as extending the confidence ellipsoids of \citet{abbasi2011improved}, which hold only in the sub-Gaussian setting.  In Subsection~\ref{subsec:emp_bern}, we prove a multivariate, self-normalized empirical Bernstein inequality, generalizing a result of \citet{howard2021time} to $d$-dimensional space. Lastly, in Appendix~\ref{app:var}, we specialize our regression bounds for the task of parameter estimation in vector autoregressive models (i.e.\ in the $\VAR{p}$ model).
\end{enumerate}

In sum, we provide time-uniform, self-normalized concentration inequalities for both scalar and vector-valued processes that hold under quite general noise conditions. Not only are these bounds of theoretical interest, but they are in fact applicable to common statistical tasks --- in particular those that can be framed in the online linear regression framework. 
\section{Background and Sub-$\psi$ Processes}
\label{sec:background}

In this section we discuss the key sub-$\psi$ condition leveraged in deriving self-normalized concentration results for vector-valued processes. We arrive at our vector sub-$\psi$ condition by extending the eponymous condition defined in the setting of scalar-valued processes \citep{howard2020time, howard2021time}, to high dimensional spaces. We first summarize some notation that will be used ubiquitously.

\paragraph{Notation:} Throughout, we define $\N = \{0, 1, 2, \cdots \}$ to be the set of natural numbers, which we assume to begin at 0. We let $\langle x, y \rangle = x^\top y$ denote that standard Euclidean inner product on $\R^d$. Additionally, we let $\S^{d - 1} := \{x \in \R^d : \|x\| = 1\}$ denote the unit sphere and $\B_d := \{x \in \R^d : \|x\| \leq 1\}$ the unit ball in $\R^d$. By $\calL_{+}(\R^d)$, we denote the set of all $d \times d$ positive semi-definite matrices, with $I_d \in \calL_+(\R^d)$ denoting the $d$-dimensional identity matrix. For $V \in \calL_+(\R^d)$, let $\gamma_{\max}(V)$ denote the largest eigenvalue of $V$, $\gamma_{\min}(V)$ the smallest eigenvalue of $V$, and let \[\kappa(V) := \frac{\gamma_{\max}(V)}{\gamma_{\min}(V)}\] denote the condition number of $V$. Each such $V$ admits a spectral decomposition of the form $V = \sum_{n = 1}^d \gamma_n(V) v_n v_n^\top$, where $(\gamma_n(V))_{n \in [d]}$ is the non-increasing sequence of eigenvalues associated with matrix $V$ and $(v_n)_{n \in [d]}$ is the corresponding sequence of unit eigenvectors, which we know forms an orthonormal basis for $\R^d$. For $\rho > 0$, let \[V \lor \rho I_d := \sum_{n = 1}^d(\gamma_n(V) \lor \rho) v_n v_n^\top,\] where for scalars $a, b \in \R$, $a \lor b := \max\{a, b\}$.

For a strictly increasing, differentiable convex function $\psi : [0, \lambda_{\max}) \rightarrow \R$ we let $\psi^\ast : [0, u_{max}) \rightarrow \R_{\geq 0}$ denote its convex conjugate, given by $\psi^\ast(u) := \sup_{\lambda \in [0, \lambda_{\max})} u\lambda - \psi(\lambda)$, where $u_{\max} := \lim_{\lambda \uparrow \lambda_{\max}}\psi'(\lambda)$. In the sequel, we will always assume $\sup_{\lambda \in [0, \lambda_{\max})}\psi'(\lambda) = \infty$, and hence will have $u_{\max} = \infty$. Some key properties of convex conjugation are that (a) $\psi^\ast$ is convex, (b) $(\psi^\ast)^\ast = \psi$, and (c) $(\psi^\ast)' = (\psi')^{-1}$.

Let $(Z, \rho)$ denote a metric space, and let $T \subset Z$. For $\epsilon > 0$, we say that a set $K \subset Z$ is an $\epsilon$-covering for $T$ if, for any $z \in T$, there exists a point $\pi(z) \in K$ satisfying $\rho(z, \pi(z)) \leq \epsilon$. We call $\pi : T \rightarrow K$ a ``projection'' onto the covering, which maps each point in $T$ onto the nearest point in $K$ (or an arbitrary one if not unique). If $K \subset T$, we call $K$ a \emph{proper} $\epsilon$-covering of $T$. We will exclusively consider proper coverings in the sequel. We define the $\epsilon$-covering number $N(T, \epsilon, \rho)$ of $T$ to be the cardinality of the smallest proper $\epsilon$-covering of $T$. Any proper $\epsilon$-covering of $T$ obtaining this minimum will be called minimal. In the special case $(Z, \rho) = (\R^d, \|\cdot\|)$ and $T = \S^{d - 1},$ we denote the $\epsilon$-covering number of $T$ by $N_{d - 1}(\epsilon)$.

Lastly, if $(S_t)_{t \geq 0}$ is some process evolving in a space $\calX$ and $t \geq 1$, we define the $t$th increment of $(S_t)_{t \geq 0}$ to be $\Delta S_t := S_t - S_{t - 1}$. If a filtration $(\calF_t)_{t \geq 0}$ is understood from context, we may use the notation $\E_t[\cdot] = \E\left( \cdot \mid \calF_t \right)$ for easing notational burden. By default, we take $\calF_0:=\{\emptyset,\Omega\}$ and $\calF_t=\sigma(S_1,\dots,S_t)$.

\paragraph{Sub-$\psi$ Processes:}

We now describe in more detail a condition that links the growth of a process $(S_t)_{t \geq 0}$ evolving in $\R^d$ to a corresponding ``accumulated variance process'' $(V_t)_{t \geq 0}$ taking values in $\calL_+(\R^d)$. This linking will occur through the consideration of a family of exponential processes in which a scaled version  of $(S_t)_{t \geq 0}$ along any fixed direction is compensated by $(V_t)_{t \geq 0}$ and a function $\psi$ that measures the heaviness of the tails of $\Delta S_t$. $\psi$ should be thought of as acting like the cumulant generating function (or CGF) of $\Delta S_t$. These exponential processes will behave like non-negative supermartingales, and thus will allow us to apply powerful time-uniform concentration results to bound the growth of an appropriately normalized version of $(S_t)_{t \geq 0}$. Due to the central role of $\psi$ in connecting the growth of $(S_t)_{t \geq 0}$ and $(V_t)_{t \geq 0}$, we will adopt the terminology of \citet{howard2020time, howard2021time} from the scalar case and refer to the condition as the ``sub-$\psi$ condition''.  

\begin{figure}
    \centering
    \subfloat[Implications amongst $\psi$]{
        \includegraphics[width=0.5\textwidth]{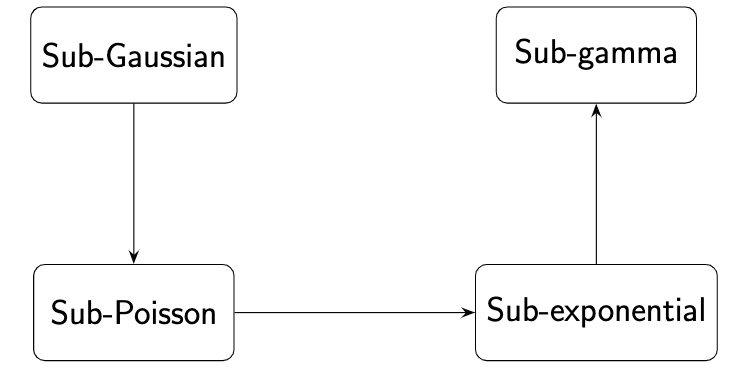}
    }
    \subfloat[Plotted $\psi$ for $c = 1$]{
        \includegraphics[width=0.5\textwidth]{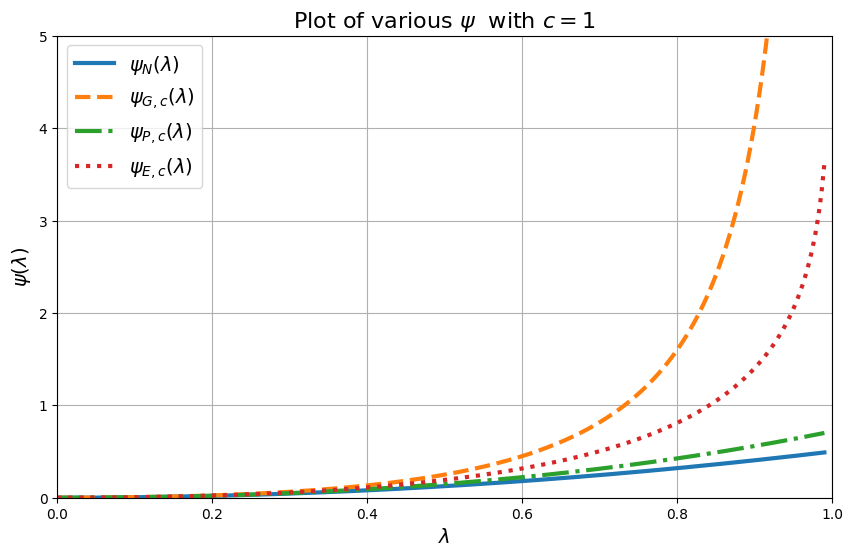}
    }
    \caption{Comparing the four CGF-like functions $\psi_N, \psi_{P, c}, \psi_{E, c}$, and $\psi_{G, c}$ discussed throughout this section. The first figure illustrates implications amongst sub-$\psi$ processes: sub-$\psi_{N} \Rightarrow$ sub-$\psi_{P, c} \Rightarrow$ sub-$\psi_{E, c} \Rightarrow$ sub-$\psi_{G, c}$. That is, of all the CGFs considered, $\psi_N$ represents the lightest tails and $\psi_{G, c}$ the heaviest --- sub-Gaussian processes are sub-Gamma but not vice versa. The second figure illustrates this by plotting $\psi(\lambda)$ for $\lambda \in [0, 1)$ and with $c = 1$.}
    \label{fig:psi_compare}
\end{figure}

Before formally defining the sub-$\psi$ condition, we must briefly discuss the properties of the function $\psi$ we will consider.

\begin{definition}
\label{def:cgf_like}
We say a function $\psi : [0, \lambda_{\max}) \rightarrow \R_{\geq 0}$ is CGF-like if $\psi$ is strictly convex,  $\psi(0) = \psi'(0) = 0$, and  $\psi''(0) > 0$.

\end{definition}

Notable examples of CGF-like functions include $\psi_N(\lambda) := \frac{\lambda^2}{2}$, the CGF of a standard normal random variable; \[\psi_{E, c}(\lambda) := \frac{-\log(1 - c\lambda) - c\lambda}{c^2},\] the CGF of a (centered) exponential random variable; $\psi_{P, c}(\lambda) := \frac{e^{c\lambda} - c\lambda - 1}{c^2}$, the CGF of a centered Poisson random variable; and \[\psi_{G, c}(\lambda) := \frac{\lambda^2}{2(1 - c\lambda)},\] a bound on the CGF of a centered Gamma random variable. Note that, in particular, $\psi_{G, 0} = \psi_N$. In general, the parameter $c$ used above helps capture the heaviness of the tails --- the larger $c$ is the heavier the tails of the observations are. We will leverage the aforementioned four CGFs in the sequel, providing explicit examples. We provide a plotted comparison amongst them in Figure~\ref{fig:psi_compare}. Basic theory regarding CGF-like functions is discussed in detail in Appendix~\ref{app:cgf}. While we will use many nontrivial properties of CGF-like functions freely hereinafter,  we will always make the proper forward reference to Appendix~\ref{app:cgf}.

We now present the sub-$\psi$ condition for scalar processes, and later for vector processes. First introduced in \citet{howard2020time}, the sub-$\psi$ condition very heuristically states that, for each $t \geq 0$, the cumulant generating function for $S_t$ is dominated by $V_t\cdot\psi$, where $\psi$ is some CGF-like function per the above definition. More precisely, the sub-$\psi$ condition for scalar-valued processes is as follows.

\begin{definition}
\label{def:psi_scal}
Let $\psi : [0, \lambda_{\max}) \rightarrow \R_{\geq 0}$ be CGF-like, let $(S_t)_{t \geq 0}$ and $(V_t)_{t \geq 0}$ be respectively $\R$-valued and $\R_{\geq 0}$-valued processes adapted to some filtration $(\calF_t)_{t \geq 0}$. We say that $(S_t, V_t)_{t \geq 0}$ is sub-$\psi$ (or equivalently that $(S_t)_{t \geq 0}$ is a sub-$\psi$ process with variance proxy $(V_t)_{t \geq 0}$) if for every $\lambda \in [0, \lambda_{\max})$, the exponential process $\exp\left\{ \lambda S_t - \psi(\lambda) V_t \right\}$ is (almost surely) upper bounded by some non-negative supermartingale $(L_t^{\lambda})_{t \geq 0}$ with respect to $(\calF_t)_{t \geq 0}$:
\[
M_t^{\lambda} := \exp\left\{ \lambda S_t - \psi(\lambda) V_t \right\} \leq L_t^{\lambda}, \qquad \text{for all } t \geq 0.
\]
\end{definition}

As an easy example, consider the case where $(X_t)_{t \geq 1}$ is a sequence of i.i.d.\ mean zero random variables with CGF $\psi(\lambda) = \log\E e^{\lambda X_1}$. Letting $S_t := \sum_{s = 1}^t X_s$ and $V_t := t$, it is easy to see that $M_t^{\lambda}$ is a non-negative martingale with respect to the natural filtration generated by the $X_t$'s (and thus we can take $L_t^{\lambda} = M_t^{\lambda}$). Definition~\ref{def:psi_scal} generalizes the above example to a setting where the random variables may have more  complicated dependence structures, and ``nonparametric'' tail conditions, including settings where $V_t$ can itself be adapted to $(\calF_t)_{t \geq 0}$ (as opposed to a constant or predictable variance processes), a key ingredient in self-normalized bounds. Recently, \citet{howard2020time} compiled a rich selection of examples of such sub-$\psi$ processes. We discuss further examples below.

The above definition for scalar-valued processes suggests a straightforward means of generalizing the sub-$\psi$ condition to the setting where $(S_t)_{t \geq 0}$ is $\R^d$-valued and $(V_t)_{t \geq 0}$ is $\calL_+(\R^d)$-valued. Namely, $(S_t, V_t)_{t \geq 0}$ should be sub-$\psi$ if the projection along any direction vector $\nu \in \S^{d- 1}$ is sub-$\psi$ in the scalar sense.

\begin{definition}
\label{def:psi_vec}
Let $\psi : [0, \lambda_{\max}) \rightarrow \R_{\geq 0}$ be CGF-like, and let $(S_t)_{t \geq 0}$ and $(V_t)_{t \geq 0}$ be respectively $\R^d$-valued and $\calL_+(\R^d)$-valued processes adapted to some filtration $(\calF_t)_{t \geq 0}$. We say that $(S_t, V_t)_{t \geq 0}$ is sub-$\psi$ if, for every $\nu \in \S^{d - 1}$, the projected process $(\langle \nu, S_t\rangle, \langle \nu, V_t \nu \rangle)_{t \geq 0}$ is sub-$\psi$ in the sense of Definition~\ref{def:psi_scal}. In other words, $(S_t, V_t)_{t \geq 0}$ is sub-$\psi$ if, for any $\nu \in \S^{d - 1}$ and $\lambda \in [0, \lambda_{\max})$, there is a non-negative supermartingale $(L_t^{\lambda \cdot \nu})_{t \geq 0}$ with respect to $(\calF_t)_{t \geq 0}$ such that
\[
M_t^{\lambda \cdot \nu} := \exp\left\{\lambda \langle \nu, S_t \rangle - \psi(\lambda)\langle \nu, V_t \nu \rangle \right\} \leq L_t^{\lambda \cdot \nu}, \qquad \text{for all } t \geq 0.
\]
\end{definition}

It is straightforward to confirm that if $(S_t, V_t)_{t \geq 0}$ is sub-$\psi$, then $(S_t, V_t + \rho I_d)_{t \geq 0}$ and $(S_t, V_t \lor \rho I_d)_{t \geq 0}$ are sub-$\psi$ as well. Furthermore, it is also straightforward to check that the rescaled process $(S_t/\sqrt{\rho}, V_t/\rho)_{t \geq 0}$ is sub-$\psi_\rho$, where $\psi_\rho : [0, \sqrt{\rho}\lambda_{\max}) \rightarrow \R_{\geq 0}$ is given by \[\psi_\rho(\lambda) := \rho\psi(\lambda/\sqrt{\rho}).\] 
These transformations are important as they will allow us to exclusively study processes satisfying $V_1 \geq 1$ in the sequel. For the sake of completeness, we prove that $\psi_\rho$ is in fact CGF-like in Proposition~\ref{prop:cgf_rescale} in Appendix~\ref{app:cgf}. We codify the above observations into the following proposition for ease of reference.

\begin{prop}
    \label{prop:sub_psi}
    Suppose $(S_t, V_t)_{t \geq 0}$ is sub-$\psi$ with (inherently with respect to some filtration $(\calF_t)_{t \geq 0}$). Then, for any fixed $\rho >0$,
    \begin{enumerate}
        \item $(S_t, V_t + \rho I_d)_{t \geq 0}$ is sub-$\psi$ with respect to $(\calF_t)_{t \geq 0}$,
        \item $(S_t, V_t \lor \rho I_d)_{t \geq 0}$ is sub-$\psi$ with respect to $(\calF_t)_{t \geq 0}$,
        and
        \item $(S_t/\sqrt{\rho}, \rho^{-1}V_t)_{t \geq 0}$ is sub-$\psi_\rho$ with respect to $(\calF_t)_{t \geq 0}$, where $\psi_\rho(\lambda) := \rho\psi(\lambda/\sqrt{\rho})$.
    \end{enumerate}
\end{prop}

As we will see, Definition~\ref{def:psi_vec} will prove to be the ``right'' generalization of the sub-$\psi$ condition to high-dimensional settings. In more detail, from the condition, we will derive a general, time-uniform bound on the self-normalized process $(\|V_t^{-1/2}S_t\|)_{t \geq 0}$ that will be tight up to small, multiplicative constants.

\paragraph{Examples of Sub-$\psi$ Processes:} We now provide several practically-relevant examples of multivariate sub-$\psi$ processes. We first provide four examples for ``light-tailed'' processes, i.e.\ processes that have slowly growing moments. In particular, we provide one example for each of the aforementioned CGF-like functions $\psi_N, \psi_P, \psi_{E, c},$ and $\psi_{G, c}$. In each of the examples below, we assume we are studying some process $(X_t)_{t \geq 1}$ that is adapted to some filtration $(\calF_t)_{t \geq 0}$. We let $(S_t)_{t \geq 0}$ be the cumulative sum process $S_t = X_1 + \cdots + X_t$ unless otherwise stated. The following examples are light-tailed and their proofs are standard:
\begin{enumerate}
    \item If $\|X_t\| \leq C_t$ almost surely where $C_t \in \calF_{t - 1}$, a standard Hoeffding-style argument yields that $S_t$ is sub-$\psi_N$ with variance proxy $V_t = \sum_{s = 1}^t C_s^2 I_d$.
    \item If $\|X_t\| \leq c$ almost surely, a standard Bennett-style argument (see the proof of Theorem 2.9 in \citet{boucheron2013concentration}) shows that $S_t := \sum_{s = 1}^t \left\{X_s - \E_{s - 1}X_s\right\}$ is sub-$\psi_{P, c}$ with variance proxy $V_t := \sum_{s = 1}^t \E_{s - 1}X_sX_s^\top$.
    \item As revisited in Section~\ref{subsec:emp_bern}, if $\|X_t\| \leq 1/2$ almost surely\footnote{Note that the the assumption $\|X_t\| \leq \frac{1}{2}$ can be replaced by any constant by appropriately changing the scale parameter of the sub-Exponential CGF.}, then $S_t  := \sum_{s = 1}^t\left\{X_s - \E_{s - 1}X_s\right\}$ is sub-$\psi_{E, 1}$ with variance proxy $V_t := \sum_{s = 1}^t(X_s - \wh{\mu}_{s - 1})(X_s - \wh{\mu}_{s - 1})^\top$. In the above, $\wh{\mu}_t := \frac{1}{t}\sum_{s = 1}^t X_s$ is the time-average mean given the first $t$ samples. From this condition, one can derive a multivariate, self-normalized ``empirical Bernstein'' inequality. This type of inequality is useful in statistical applications~\citep{waudby2020estimating} due to the fact its tightness adapts to the observed (i.e.\ empirical) variance within the samples witnessed. 
    \item Lastly, if $\E_{t - 1}|\langle \nu, X_t\rangle|^k \leq \frac{k!}{2}c^{k - 2}\E_{t - 1}\langle \nu, X_t\rangle^2$ for all directions $\nu \in \S^{d - 1}$ and some constant $c > 0$, a standard application of the Bernstein condition in each direction $\nu \in \S^{d - 1}$ (see Theorem 2.10 of \citet{boucheron2013concentration}) yields that $S_t := \sum_{s = 1}^t \left\{X_s - \E_{s - 1}X_s\right\}$ is sub-$\psi_{G, c}$ with variance proxy $V_t := \sum_{s = 1}^t \E_{s - 1}X_s X_s^\top$.

\end{enumerate}

Processes of the above form appear in many tasks in statistics, computational learning theory, and theoretical computer science. However, self-normalized concentration also allows one to move beyond light-tailed settings to prove concentration of measure results for processes lacking finite moments of all orders. We provide several examples of these ``heavy-tailed'' processes below:
\begin{enumerate}
    \item If $X_t =_d -X_t \mid \calF_{t - 1}$ (that is, the $X_t$ are conditionally symmetric), Lemma 3 of \citet{de2007pseudo} can be used to show that $S_t$ is sub-$\psi_N$ with variance proxy $V_t := \sum_{s = 1}^t X_sX_s^\top$. This provides salient example of how the sub-$\psi$ condition can be leveraged to provide meaningful concentration for processes whose increments may even lack a well-defined mean (e.g. take the $X_t$ to be i.i.d.\ Cauchy random variables).
    \item If $\E_{t - 1}\langle \nu, X_t\rangle^2 < \infty$ for all $t$ and $\nu \in \S^{d - 1}$, then Lemma 3 of \citet{howard2020time} can be used to show $S_t$ is sub-$\psi_N$ with variance proxy $V_t = \frac{1}{3}\sum_{s = 1}^t X_sX_s^\top + \frac{2}{3}\sum_{s = 1}^t \E_{s - 1}X_s X_s^\top$.
    \item Finally, if one further assumes that $\E_t |\langle \nu, X_t\rangle|^3$ is almost surely finite for all $t$ and $\nu$, one can show that $S_t$ is sub-$\psi_{G, 1/6}$ with variance proxy $V_t = \sum_{s = 1}^t\left\{ X_s X_s^\top + \E_{s - 1}\|X_s\|^3 I_d\right\}$. We also show this in Appendix~\ref{app:lems}.
\end{enumerate}

While the list of light and heavy-tailed examples above is not exhaustive, it illustrates the generality of the vector sub-$\psi$ condition presented above. As a consequence of our main theorem (Theorem~\ref{thm:vector}), one directly arrive at non-trivial concentration of measure results for each of the above examples. Further, even in settings where one can apply the method of mixtures result due to de la Pe\~na (such as in the case of symmetric observations above) our result provides distinct, generally incomparable rates.

\paragraph{Super-Gaussian CGFs:} Lastly, we draw attention to \textit{super-Gaussian} CGF-like functions $\psi$:
\begin{quote}
a CGF-like function $\psi$ is super-Gaussian if $\frac{\psi(\lambda)}{\lambda^2}$ is an increasing function of $\lambda$. 
\end{quote}
In words, $\psi$ is super-Gaussian if it grows at least as rapidly as $\psi_N$, the CGF of a $\calN(0, 1)$ random variable. Most notable examples of CGF-like functions are super-Gaussian, with particularly important examples being $\psi_N, \psi_{E, c}, \psi_{G, c},$ and $\psi_{P, c}$. Informally, one typically needs to use a super-Gaussian CGF if the underlying random process is heavier tailed than a sub-Gaussian process.

One example of a CGF that is not super-Gaussian would be $\psi_{B, p}(\lambda)$, the CGF of a centered Bernoulli random variable $X$ with $\P(X = 1) = p$. We discuss equivalent definitions and properties of CGF-like functions in detail in Appendix~\ref{app:cgf}. While our bounds will hold in the case where $(S_t, V_t)_{t \geq 0}$ is sub-$\psi$ for arbitrary $\psi,$ they are particularly clean when $\psi$ is super-Gaussian, and we emphasize this case going forward.

\section{A General Non-Asymptotic LIL for Scalar Processes}
\label{sec:scalar}
In this section, we prove a high-probability, time-uniform bound on the growth of a scalar process $(S_t)_{t \geq 0}$ normalized by some measure of accumulated variance $(V_t)_{t \geq 0}$. In particular, in Theorem~\ref{thm:scalar} below, we show that if $(S_t, V_t)_{t \geq 0}$ is a sub-$\psi$ process, then, with high probability, simultaneously for all $t \geq 0$,
\[
S_t \lesssim V_t \cdot (\psi^\ast)^{-1}\left(\frac{1}{V_t}\log\log(V_t)\right),
\]
where we have omitted dependence on several user-chosen parameters and constants for the sake of exposition. As will be seen in the sequel, all such constants are small. Dividing both sides by $\sqrt{V_t}$ yields a result in ``self-normalized'' form that looks more akin to the results in subsequent sections, but we adopt the above form for consistency with existing results \citep{howard2020time, howard2021time}.
Since $(\psi^\ast)^{-1}(u) \sim \sqrt{2u}$ as $u \downarrow 0$ whenever $\psi(\lambda) \sim \frac{\lambda^2}{2}$ as $\lambda \downarrow 0$ (as is the case for all CGF-like functions addressed in the previous section), for large values of $V_t$, the above high probability bound can be written as 
\[
S_t \lesssim \sqrt{V_t\log\log(V_t)},
\]
thus allowing our results in this section to be viewed as a non-asymptotic (i.e.\ finite sample) version of the law of the iterated logarithm. We further describe connections between our scalar-valued bound and the law of the iterated logarithm in Subsection~\ref{subsec:scalar_LIL} below.

While we construct the bounds in this section as a requisite for deriving self-normalized concentration inequalities for vector-valued processes, we believe the results are of independent interest. In particular, our results are significantly more general than those of \citet{howard2021time}, which only hold for sub-$\psi_{G, c}$ (i.e. sub-Gamma) processes. While \citet{howard2020time} show that any CGF-like function $\psi$ can be bounded point-wise by $a \psi_{G, c}$ for appropriately chosen constants $a, c > 0$, this comparison can be arbitrarily loose, especially for small values of $V_t$. 
We illustrate this in Figure~\ref{fig:bdry_poisson} in Appendix~\ref{app:figs}, which shows that the time-uniform boundary presented in Theorem~\ref{thm:scalar} (applied in the sub-Poisson setting $\psi = \psi_{P, c}$) can offer improved concentration over the main theorem of \citet{howard2021time}
We further discuss comparisons between our bounds and those of \citet{howard2021time} following the proof of Theorem~\ref{thm:scalar}.

\begin{theorem}
\label{thm:scalar}
Suppose $(S_t, V_t)_{t \geq 0}$ is a real-valued sub-$\psi$ process for some CGF-like function $\psi : [0, \lambda_{\max}) \rightarrow \R_{\geq 0}$ satisfying $\lim_{\lambda \uparrow \lambda_{\max}}\psi'(\lambda) = \infty$. Let $\alpha > 1, \rho > 0,$ and $\delta \in (0, 1)$ be constants respectively representing the stitching epoch length, the minimum intrinsic time, and the error probability. Let $h : \R_{\geq 0} \rightarrow \R_{\geq 0}$ be an increasing function such that $\sum_{k \in \N}h(k)^{-1} \leq 1$, representing how the error is spent across epochs. Define the function $\ell_\rho : \R_{\geq 0} \rightarrow \R_{\geq 0}$ by
\[
\ell_\rho(v) = \log\left(h\left(\log_\alpha\left(\frac{v \lor \rho}{\rho}\right)\right)\right) + \log\left(\frac{1}{\delta}\right),
\]
where we have suppressed the dependence of $\ell_\rho(v)$ on $\alpha, h$ for brevity.
Then, we have
\[
\P\left(\exists t \geq 0 : S_t \geq (V_t \lor \rho) \cdot (\psi^\ast)^{-1}\left(\frac{\alpha}{V_t \lor \rho}\ell_\rho(V_t)\right)\right) \leq \delta.
\]

\end{theorem}

We provide a full proof of Theorem~\ref{thm:scalar} in Section~\ref{sec:proof} below. Except for the unavoidable error probability $\delta$, we briefly elaborate on the other user-specified constants that appear in the statement of the theorem:
\begin{enumerate}
    \item $\alpha > 1$ controls the spacing of ``intrinsic time'' or accumulated variance of the process $(S_t)_{t \geq 0}$. Heuristically, Theorem~\ref{thm:scalar} will be obtained by optimizing tight, linear boundaries on events of the form $\{ \alpha^k \leq V_t < \alpha^{k + 1} \}$.
    \item $\rho > 0$ gives the first  ``intrinsic time'' at which our boundaries start depending on the variance process $(V_t)_{t \geq 0}$. When $0 \leq V_t < \rho$, the boundary will only depend on $\rho$.
    \item $h : \R_{\geq 0} \rightarrow \R_{>0}$ is a function satisfying $\sum_{k \geq 0}h(k)^{-1} \leq 1$. $h$ defines how much of the overall probability mass associated with failure (determined by $\delta$) to allocate to each event of the form $\{ \alpha^k \leq V_t < \alpha^{k + 1}\}$.
\end{enumerate}

In the above, we view the parameters $\rho,$ and $h$ as critical, since they directly affect the shape and validity of the bound, whereas we view $\alpha$ as less critical, as any small variation in $\alpha$ will only minimally affect the tightness of the bound in terms of constants. For example, a reasonable choice of this temporal spacing parameter is $\alpha = 1.05$. \citet{howard2021time} discuss reasonable choices for the function $h$, and we emphasize in the sequel the choice of $h(k) := (k + 1)^s \zeta(s)$, where $s > 1$ is a tuning parameter and $\zeta$ is the Riemann zeta function. This choice is of particular theoretical interest as it yields non-asymptotic rates that depend on $\log\log(V_t)$ (up to constants), thus allowing our bound to be viewed as a general, non-asymptotic version of the LIL. We in particular use this choice of $h$ in the proof of Corollary~\ref{cor:lil_scalar} in Subsection~\ref{subsec:scalar_LIL} below. We briefly state a corollary of the above theorem in the case of sub-Gamma processes, which may be of particular practical interest.

\begin{corollary}
\label{cor:sub_gamma_scalar}
Assume the same setup as in Theorem~\ref{thm:scalar}, and further suppose that (a) $\psi = \psi_{G, c}$ and (b) $h(k) \leq Ak^B$ for some constants $A, B$ and all $k \geq 1$. Then, with probability $\geq 1 - \delta$, simultaneously for all $t \geq 0$ such that $V_t \geq \rho$
\begin{align*}
S_t &\leq \sqrt{2\alpha V_t \left[B\log\left(A\log_\alpha\left(\frac{V_t}{\rho}\right)\right) + \log\left(\frac{1}{\delta}\right)\right]} + c\alpha \left[ B\log\left(A\log_\alpha\left(\frac{V_t}{\rho}\right)\right) + \log\left(\frac{1}{\delta}\right)\right] \\
&\lesssim \sqrt{V_t\left[\log\log(V_t) + \log(1/\delta)\right]} + c\left[\log\log(V_t) + \log(1/\delta)\right]
\end{align*}
\end{corollary}

We sketch our proof of Theorem~\ref{thm:scalar} here to illustrate how we are able to generalize the results of \citet{howard2021time}. Much like the ``stitching'' technique of the aforementioned authors, our argument proceeds by breaking ``intrinsic'' time into geometric epochs of the form $\{\alpha^k \leq V_t < \alpha^{k + 1}\}$ and then optimizing a tight linear inequality in each period. 
The key difference is how we optimize this boundary for $(S_t)_{t \geq 0}$. The techniques leveraged by \citet{howard2021time} yield a boundary that is defined in terms of $\psi_{G, c}^{-1}$. From our understanding of the Chernoff method, we know that if a mean zero random variable $X$ has associated CGF $\psi(\lambda) := \log\E e^{\lambda X}$, then we have $\P\left(X \geq (\psi^\ast)^{-1}\left(\log\left(\frac{1}{\delta}\right)\right)\right) \leq \delta.$ Thus, we at the very least expect to obtain a boundary defined in terms of $(\psi^\ast)^{-1}$. By leveraging the ``slope transform'' of $\psi$ (detailed in Appendix~\ref{app:cgf}), we are able to obtain an inequality for \emph{any} sub-$\psi$ processes with a surprisingly straightforward argument. To our knowledge, we are the first to show a connection between this transformation and non-asymptotic laws of the iterated logarithm.

\subsection{A Detailed Comparison With Existing Bounds}
\label{subsec:scalar:compare}
Theorem~\ref{thm:scalar} can further be compared to Theorem 1 of \citet{howard2021time}, who provide time-uniform, self-normalized concentration for scalar processes in the sub-$\psi_{G, c}$ case. In particular, the theorem from \citet{howard2021time} yields the following high probability time-uniform bound:
\begin{equation}
\label{eq:bdry_gamma}
S_t \leq \sqrt{\left(\frac{\alpha^{1/4} + \alpha^{-1/4}}{\sqrt{2}}\right)^2(V_t \lor \rho)\ell_\rho(V_t) + c^2\left(\frac{\sqrt{\alpha} + 1}{2}\right)^2\ell_\rho(V_t)^2} + c\left(\frac{\sqrt{\alpha} + 1}{2}\right)\ell_\rho(V_t).
\end{equation}
Using our bound (Theorem~\ref{thm:scalar}) we instead obtain:
\[
S_t \leq \sqrt{2\alpha (V_t \lor \rho)\ell_\rho(V_t)} + c\alpha \ell_\rho(V_t).
\]

For $c > 0$, the bound from \citet{howard2021time} suffers from an additional additive term of $c^2\left(\frac{\sqrt{\alpha} + 1}{2}\right)^2 \ell_\rho(V_t)^2$ inside the square root. This is due to the authors optimizing their boundary in terms of $\psi_{G, c}^{-1}$. By directly optimizing our boundary in terms of $(\psi_{G, c}^\ast)^{-1}$ (as suggested by the Chernoff method) we are able to avoid this dependence. We do note that in the sub-Gaussian setting ($c = 0$), the bound in Equation~\eqref{eq:bdry_gamma}
is (slightly) tighter than our own, as $\left(\frac{\alpha^{1/4} + \alpha^{-1/4}}{\sqrt{2}}\right)^2 \leq 2\alpha$. However, for $\alpha < 2.06,$ we have $2\alpha \leq 2\left(\frac{\alpha^{1/4} + \alpha^{-1/4}}{\sqrt{2}}\right)^2,$ showing that our bounds are looser than those of \citet{howard2021time} by a multiplicative factor of no more than $\sqrt{2}$ in this regime. In particular, as $\alpha$ is decreased towards 1, the multiplicative factor by which our bounds are suboptimal to those of \citet{howard2021time} vanishes to 1. 

We emphasize that the bounds of \citet{howard2021time} hold \textit{only} in the sub-Gamma case. While sub-Gamma concentration can be applied to sums of sub-Exponential and sub-Poisson random variables, this approximation is far from tight, especially in small sample sizes. Our results hold directly for \textit{any} CGF-like function $\psi$, including all listed in Section~\ref{sec:background}.

\subsection{Asymptotic Law of the Iterated Logarithm}
\label{subsec:scalar_LIL}

In the preceding paragraphs, we derived time-uniform bounds for general scalar-valued sub-$\psi$ processes. In particular, we argued our presented results generalized those of \citet{howard2021time}, who show a similar result for the case $\psi = \psi_{G, c} = \frac{\lambda^2}{2(1 - c \lambda)}$ (i.e.\ when $\psi$ is the CGF-like function associated with a sub-Gamma random variable). As noted above, for any fixed step size $\alpha > 1$, in the case $c = 0$ (i.e. when $\psi = \psi_N$ is the CGF of a standard Gaussian random variable), the bounds of \citet{howard2021time} dominate ours, albeit by a vanishingly small multiplicative factor as $\alpha \downarrow 1$. 

This begs the following question: are our bounds ``optimal'' in the sense that, by appropriately selecting the tuning parameters, they recover the asymptotic (upper) law of the iterated logarithm with the correct constant. In Corollary~\ref{cor:lil_scalar} below, we show that this exactly the case, and thus derive a law of the iterated logarithm for sub-$\psi$ processes.

\begin{corollary}
\label{cor:lil_scalar}
Let $(S_t)_{t \geq 0}$ be sub-$\psi$ with variance proxy $(V_t)_{t \geq 0}$, and suppose that $\psi''(\lambda)= 0$ and $V_t \xrightarrow[t \rightarrow \infty]{} \infty$ almost surely. Then,
\[
\limsup_{t \rightarrow \infty}\frac{S_t}{\sqrt{2V_t\log\log(V_t)}} \leq 1 \;\; \text{almost surely}.
\]
\end{corollary}

Corollary~\ref{cor:lil_scalar} follows as a direct consequence of Corollary~\ref{cor:lil_vec}, which provides an asymptotic law of the iterated logarithm for vector-valued sub-$\psi$ processes, noting that the dependence of the bound on the condition number of $V_t$ vanishes in the scalar case. 

\section{Main result}
\label{sec:vector}

We now present the main result of this paper: a time-uniform, self-normalized concentration inequality for a general class of processes evolving in $\R^d$. In short, while our results in the previous section could be seen as controlling the growth of the process $(S_t)_{t \geq 0}$ by stitching over various scales of ``intrinsic time'', our result in this section follows from stitching over various scales of intrinsic geometric distortion. This is ultimately controlled by the condition number $\kappa(V_t)$.

\begin{theorem}
\label{thm:vector}
Suppose $(S_t)_{t \geq 0}$ is a sub-$\psi$ process with variance proxy $(V_t)_{t \geq 0}$ taking values $\R^d$. Let $\alpha > 1$, $\beta > 1$, $\rho > 0$, $\epsilon \in (0, 1),$ and $\delta \in (0, 1)$ be constants, and let $h : \R_{\geq 0} \rightarrow \R_{\geq 0}$ be an increasing function such that $\sum_{k \in \N}h(k)^{-1} \leq 1$. Define the function\footnote{Recall $N_{d - 1}(\epsilon)$ was defined to be the $\epsilon$-covering number of $\S^{d - 1}$.} $L_\rho : \calS_+^d \rightarrow \R_{\geq 0}$ by
\begin{align*}
L_\rho(V) &:= \log\left(h\left(\log_\alpha\left(\frac{\gamma_{\max}(V \lor \rho I_d)}{\rho}\right)\right)\right) + \log\left(\frac{1}{\delta}\frac{1}{1 - \beta^{-1}}\right) \\
&+ \log\left(\beta \sqrt{\kappa(V \lor \rho I_d)} \cdot N_{d - 1}\left(\frac{\epsilon}{\beta\sqrt{\kappa(V\lor \rho I_d)}}\right)\right).
\end{align*}
If $\psi$ is super-Gaussian (meaning $\psi(\lambda)/\lambda^2$ is an increasing function of $\lambda$), then
\[
\P\left(\exists t \geq 0 : \left\|(V_t \lor \rho I_d)^{-1/2} S_t\right\| \geq \frac{\sqrt{\gamma_{\min}(V_t \lor \rho I_d)}}{1 - \epsilon}\cdot(\psi^\ast)^{-1}\left(\frac{\alpha}{\gamma_{\min}(V_t \lor \rho I_d)}L_\rho(V_t)\right)\right) \leq \delta.
\]

\end{theorem}

In addition to the parameters $\alpha, \rho,$ and $h$ from  Theorem~\ref{thm:scalar}, there are two new user-specified constants that govern the geometric aspects of our bound presented in Theorem~\ref{thm:vector}.

\begin{enumerate}
    \item $\beta > 1$ controls the spacing of how the action of the sequence of matrices $(V_t)_{t \geq 0}$ distorts the geometry of $\R^d$. Heuristically, Theorem~\ref{thm:vector} will be obtained by optimizing self-normalized inequalities on events of the form $\{\beta^k \leq \sqrt{\kappa(V_t)} < \beta^{k + 1}\}$ and carefully performing a union bound.
    \item $\epsilon \in (0, 1)$ controls the ``mesh'' or level of granularity at which we approximate the geometry of the unit sphere $\S^{d - 1}$ in the covering argument we make.
\end{enumerate}

In the vocabulary of our preceding results, we view neither $\beta$ nor $\epsilon$ as being critical parameters in optimizing our boundary. In particular, for simplicity, reasonable default choices would be $\beta = 2$ and $\epsilon = \frac{1}{2}$.

We now discuss the intuition for our argument ---  a full proof for Theorem~\ref{thm:vector} can be found in Section~\ref{sec:proof}. Our results follow by coupling our scalar-valued self-normalized inequalities, presented in the previous section, with a careful, geometric argument. If we wanted to control the un-normalized quantity $\|S_t\|$, we could just construct a finite cover of $\S^{d - 1}$ and control the growth along each direction using the scalar bounds. However, controlling $\|V_t^{-1/2}S_t\|$ is less straightforward. Using a careful transformation of variables, we can relate the magnitude of $\|V_t^{-1/2}S_t\|$ to the supremum of $\frac{\langle \nu, S_t\rangle}{\sqrt{\langle \nu, V_t \nu\rangle}}$ over a carefully-chosen, finite collection of points. In particular, the cardinality of this set of points must be selected according to how much $V_t$ ``distorts'' the geometry of $\R^d$ --- when $\kappa(V_t)$ is small we only need a small number of points, and when $\kappa(V_t)$ is large we need many points. Thus, we must maintain a nested sequence of increasingly-fine coverings of $\S^{d - 1}$ and select the one whose mesh reflects the current level of distortion.

We comment that a result similar to the above holds even in the setting where the CGF-like function $\psi$ is not super-Gaussian. In particular, en route to proving the above, we will show that, if $(S_t, V_t)_{t \geq 0}$ is sub-$\psi$ for any CGF-like $\psi$, we have
\[
\P\left(\exists t \geq 0 : \left\|(V_t \lor \rho I_d)^{-1/2} S_t\right\| \geq \sup_{\nu \in \S^{d -1}}\frac{\sqrt{\langle \nu, V_t \nu \rangle}}{1 - \epsilon}\cdot(\psi^\ast)^{-1}\left(\frac{\alpha}{\langle \nu, V_t \nu\rangle}L_\rho(V_t)\right)\right) \leq \delta.
\]
The assumption that $\psi$ is super-Gaussian merely allows us to compute the supremum over $\nu \in \S^{d - 1}$ in the above expression, giving the result a cleaner form. This assumption is not restrictive, as many reasonable examples of CGF-like functions are super-Gaussian (e.g. $\psi_N$, $\psi_{G, c}$, $\psi_{E, c}$, and $\psi_{P, c}$ to name a few encountered earlier).

To simplify Theorem~\ref{thm:vector} further we can plug in upper bounds on $N_{d - 1}(\epsilon)$ into Theorem~\ref{thm:vector}. In Lemma~\ref{lem:cov_linf} in Appendix~\ref{app:lems}, we show that $N_{d - 1}(\epsilon) \leq C_d\left(\frac{3}{e}\right)^{d - 1}$ for some constant $C_d$ that does not depend on $\epsilon$. Likewise, one can plug in the classical bound $N_{d - 1}(\epsilon) \leq \left(\frac{3}{\epsilon}\right)^d$ (which follows from Lemma 5.7 of \citet{wainwright2019high}). Treating the tuning parameters $\alpha, \beta, \epsilon, \rho$ as constants and selecting $h : \R_{\geq 0} \rightarrow \R_{\geq 0}$ satisfying $h(k) = O(k^s)$ for some $s > 1$ (which, as noted by \citet{howard2021time}, holds when $h(k) := (k + 1)^s\zeta(s)$), Theorem~\ref{thm:vector} yields that, with high probability, simultaneously for all $t \geq 0$,
\[
\left\|V_t^{-1/2}S_t\right\| = O\left(\sqrt{\gamma_{\min}(V_t)}\cdot(\psi^\ast)^{-1}\left(\frac{1}{\gamma_{\min}(V_t)}\left[\log\log(\gamma_{\max}(V_t)) + d\log\kappa(V_t)\right]\right)\right).
\]
We now specify the sub-Gamma (i.e.\ sub-$\psi_{G, c}$) case as an important corollary. 

\begin{corollary}
\label{cor:sub_gamma_vector}
Assume the same setup as in Theorem~\ref{thm:vector}, and further suppose that (a) $\psi = \psi_{G, c}$ and (b) $h(k) \leq Ak^B$ for some constants $A, B$ and all $k \geq 1$. Then, with probability $\geq 1 - \delta$, simultaneously for all $t \geq 0$ such that $V_t \succeq \rho I_d$, we have
\begin{align*}
&\|V_t^{-1/2}S_t\| \leq \frac{1}{1 - \epsilon}\sqrt{2\alpha \left[B\log\left(A\log_\alpha\left(\frac{\gamma_{\max}(V_t)}{\rho}\right)\right) + \log\left(\frac{1}{\delta}\frac{1}{1 - \beta^{-1}}\right) + (d + 1)\log\left(\frac{\beta \sqrt{\kappa(V_t)}}{\epsilon}\right)\right]} \\
&\qquad+ \frac{c\alpha}{\sqrt{\gamma_{\min}(V_t)}} \left[B\log\left(A\log_\alpha\left(\frac{\gamma_{\max}(V_t)}{\rho}\right)\right) + \log\left(\frac{1}{\delta}\frac{1}{1 - \beta^{-1}}\right) + (d+1)\log\left(\frac{\beta \sqrt{\kappa(V_t)}}{\epsilon}\right)\right]. 
\end{align*}
In particular, this implies for all $t \geq 1$ with $V_t \succeq \rho I_d$, 
\begin{align*}
\|V_t^{-1/2}S_t\| &\lesssim \sqrt{\log\log\gamma_{\max}(V_t) + \log(1/\delta) + d\log\kappa(V_t)} \\
&+\frac{c}{\sqrt{\gamma_{\min}(V_t)}}\left[\log\log\gamma_{\max}(V_t) + \log(1/\delta) + d\log\kappa(V_t)\right].
\end{align*}

\end{corollary}

\noindent When $c=0$, $\psi_{G, c}(\lambda) = \psi_N(\lambda) = \frac{\lambda^2}{2}$, and the above bound reduces to the form:
\begin{equation}
\label{eq:LIL_order}
\left\|V_t^{-1/2}S_t\right\| \lesssim \sqrt{\log\log(\gamma_{\max}(V_t)) + d\log\kappa(V_t) + \log(1/\delta)}.
\end{equation}
The bound~\eqref{eq:LIL_order}, in particular, captures the asymptotic growth rate of very general classes of sub-$\psi$ process when $\psi(\lambda) \sim \frac{\lambda^2}{2}$ as $\lambda \downarrow 0$ (in a sense that we will make fully precise soon).

\subsection{Comparison With Existing Bounds}
\label{subsec:compare:vec}

First, we compare our multivariate, self-normalized bounds to the ``method of mixtures'' bounds for sub-Gaussian concentration, in particular the following bound that follows from Example 4.2 of \citet{de2007pseudo} and Theorem 1 of \citet{abbasi2011improved} and has become a staple in constructing confidence sets in online learning tasks \citep{krishnamurthy2018semiparametric, whitehouse2023improved, durand2018streaming}. We rephrase their sub-Gaussian result in the setting of ``sub-$\psi_N$'' concentration to ease comparison with our results.

\begin{fact}[\citet{de2007pseudo, abbasi2011improved}]
\label{fact:mixture}
Let $(S_t, V_t)_{t \geq 0}$ be an $\R^d$-valued sub-$\psi_N$ process where $V_t = \sum_{s = 1}^t \E_{s - 1}\Delta S_s \Delta S_s^\top$. Then, for any $\delta \in (0, 1)$ and any $\rho > 0$, with probability at least $1 - \delta$, simultaneously for all $t \geq 0$,
\[
\left\|(V_t + \rho I_d)^{-1/2}S_t\right\| \leq \sqrt{2\log\left(\frac{1}{\delta}\sqrt{\det\left(I_d + \rho^{-1}V_t\right)}\right)}.
\]
\end{fact}

We note that the above bound holds \textit{only} in the case where the process $(S_t)_{t \geq 0}$ has sub-Gaussian increments, and it is not obvious whether or not a similar result holds for other tails, for more CGF-like functions $\psi$, and adapted (not predictable) $V_t$. In the case $\psi = \psi_N$, as noted in~\eqref{eq:LIL_order}, our bound is of the form $\left\|V_t^{-1/2}S_t\right\| = O\left(\sqrt{\log\log(\gamma_{\max}(V_t)) + d\log\kappa(V_t)}\right).$ These two bounds (those based on the determinant of the variance proxy and those based on the condition number of the variance proxy) are fundamentally incomparable. When $V_t$ is well-conditioned, we expect our bounds to be tighter than the bound in Fact~\ref{fact:mixture}, as our bounds will be $O(\sqrt{\log\log(\gamma_{\max}(V_t)) + d})$. If $\kappa(V_t) \approx \gamma_{\max}(V_t),$ we may expect the determinant rate bound in Fact~\ref{fact:mixture} to be tighter, as the bound provided by Theorem~\ref{thm:vector} will be  $O(\sqrt{\log\log(\gamma_{\max}(V_t)) + d\log(\gamma_{\max}(V_t))})$, and $d\log\gamma_{\max}(V_t) \geq \log\det(V_t)$ (ignoring the shift $\rho$ in the covariance matrix). One particularly useful feature of our bounds is that the do not require a shift in variance proxy as the bound in Fact~\ref{fact:mixture} does. It is an interesting open problem to derive determinant-rate self-normalized bounds under more general tail conditions and for adapted (not predictable) $V_t$. 

We can also compare our bounds to the recent  bounds constructed by \citet{manole2021sequential} using backwards or reverse martingale techniques. We note that the bounds of \citet{manole2021sequential} hold for \textit{any} fixed norm on $\R^d$ (e.g.\ $\ell_p$ norms, for instance), but we only present the result in the case of the $\ell_2$ norm, as this is the setting in which our bounds are comparable. The authors leverage the following bounds in estimating an unknown, multivariate mean from i.i.d.\ data. In our statement below, we center all observations so that the unknown mean always takes value zero for ease of comparison.

\begin{fact}[Corollary 23 of \citet{manole2021sequential}]
    Let $S_t := \sum_{s = 1}^t X_s$, where $(X_t)_{t \geq 0}$ are i.i.d.\ with mean 0. Let $h : \R_{\geq 0} \rightarrow \R_{\geq 0}$ satisfy $\sum_{k = 0}^\infty h(k)^{-1} \leq 1$, and let $\psi : [0, \lambda_{\max}) \rightarrow \R_{\geq 0}$ be CGF-like. Suppose that, for any $\lambda \in [0, \lambda_{\max})$ and $t \geq 0$, $\sup_{\nu \in \S^{d - 1}}\log\E e^{\lambda\langle \nu, X_t\rangle} \leq \psi(\lambda)$. Then, for any $\delta \in (0, 1),$ with probability at least $1 - \delta$, simultaneously for all $t \geq 0$,
    \[
    \|S_t/\sqrt{t}\| \leq \frac{\sqrt{t}}{1 - \epsilon}\cdot(\psi^\ast)^{-1}\left(\frac{2}{t}\left[\log(h(\log_2(t))) + \log\left(\frac{1}{\delta}\right) + \log N_{d - 1}(\epsilon)\right]\right).
    \]
\end{fact}

It is clear that the process $(S_t)_{t \geq 0}$ is sub-$\psi$ with variance proxy $(V_t)_{t \geq 0}$ given by $V_t := t I_d$, and so Theorem~\ref{thm:vector} (taking $\rho = 1$) applied to this setting yields that, with probability at least $1 - \delta$, simultaneously for all $t \geq 1$,
\[
\|S_t/\sqrt{t}\| \leq \frac{\sqrt{t}}{1 - \epsilon}\cdot(\psi^\ast)^{-1}\left(\frac{\alpha}{t}\left[\log(h(\log_\alpha(t))) + \log\left(\frac{\beta}{\delta(1 - \beta^{-1})}\right) + \log N_{d - 1}(\epsilon)\right]\right).
\]
In this particular setting, our bound is almost equivalent to that of \citet{manole2021sequential}, being looser is a vanishingly small additive factor $\log\left(\frac{\beta}{1 - \beta^{-1}}\right)$ due to the covering argument needed to control the geometric ``distortions'' induced by the variance proxy $(V_t)_{t \geq 0}$. However, we note that our bound is significantly more general, as it allows for arbitrary martingale dependence between observed random variables. This is in contrast to the bound of \citet{manole2021sequential}, as this bound is only valid if the data are known to be i.i.d.\ (or, at the very least, exchangeable). The argument used by \citet{manole2021sequential} does not readily generalize to general dependence structures because they leverage reverse martingales in the exchangeable filtration, thus requiring that the data be exchangeable.

\subsection{Vector Laws of the Iterated Logarithm}

In Corollary~\ref{cor:lil_scalar}, we discussed how our scalar bounds can be used to derive a version of the law of the iterated logarithm for scalar sub-$\psi$ processes. In particular, this bound obtained the optimal constant matching the case of i.i.d.\  random variables (see \citet{durrett2019probability}, Chapter 8 or \citet{howard2021time}), showing that our bounds are unimprovable asymptotically.

In the multivariate setting, our bounds do not just depend on $\log\log(\gamma_{\max}(V_t))$, but also on $\log\kappa(V_t)$. This dependence is not simply an artefact of our analysis, as \citet{de2007pseudo} show an example of a 2-dimensional process $(S_t)_{t \geq 0}$ and $(V_t)_{t \geq 0}$ satisfying $\|V_t^{-1/2}S_t\| \sim \sqrt{\log\kappa(V_t)}$ almost surely.

In this section, we aim to show that our results are asymptotically optimal in the following sense. First, we show that, under a simple set of assumptions, if $(S_t, V_t)_{t \geq 0}$ is a sub-$\psi$, then $\limsup_{t \rightarrow \infty}\frac{\|V_t^{-1/2}S_t\|}{\sqrt{2\log\log(\gamma_{\max}(V_t)) + d\log\kappa(V_t)}} \leq 1$ almost surely. Secondly, we show that this bound is ``tight'' in the sense that there exists a sub-$\psi$ process $(S_t, V_t)_{t \geq 0}$ such that $\|V_t^{-1/2}S_t\| = \Theta(\sqrt{\log\log\gamma_{\max}(V_t) + d\log\kappa(V_t)})$ almost surely.

We start by presenting the first result, which can be viewed as an ``upper law of the iterated logarithm''. We prove this result in Section~\ref{sec:proof} --- Example~\ref{eg:ddim} proves the lower bound.

\begin{corollary}
\label{cor:lil_vec}
Let $(S_t)_{t \geq 0}$ be an $\R^d$-valued sub-$\psi$ process with variance proxy $(V_t)_{t \geq 0}$. Suppose that (a) $\psi''(0) =1$, (b) $\gamma_{\min}(V_t) \xrightarrow[t \rightarrow \infty]{} \infty$ almost surely, and (c) and $\frac{\log(\gamma_{\max}(V_t))}{\gamma_{\min}(V_t)} = o(1)$ almost surely. Then,
\[
\limsup_{t \rightarrow \infty}\frac{\|V_t^{-1/2}S_t\|}{\sqrt{2\log\log\gamma_{\max}(V_t) + d\log\kappa(V_t)}} \leq 1
\]
almost surely. 
For $d>1$, there exist examples for which 
\[
\liminf_{t \rightarrow \infty}\frac{\|V_t^{-1/2}S_t\|}{\sqrt{ (d/2)\log\kappa(V_t)}} \geq 1,
\]
so that our upper bound is not improvable by more than a small constant factor.
\end{corollary}

We can compare the above corollary to the discussion at the beginning of Section 3 of \citet{victor2009theory}, where the authors show that when $(S_t)_{t \geq 0}$ and $(V_t)_{t \geq 0}$ satisfy certain assumptions based on finiteness of $p$th moments, one has
\[
\limsup_{t \rightarrow \infty}\frac{\|(V_t + V)^{-1/2}S_t\|}{\sqrt{\log\log\gamma_{\max}(V + V_t) + \log\kappa(V_t + V)}} < \infty \quad\text{almost surely}.
\]
Our bound is more precise than their bound in that (a) we obtain an explicit constant in our asymptotic bound, (b) the bound exactly recovers the LIL in the case $d = 1$ (see the earlier discussed Corollary~\ref{cor:lil_scalar}), and (c) our bound elicits explicit dependence on the ambient dimension $d$.

The remaining question is if the above law of the iterated logarithm is tight. As aforementioned, \citet{de2007pseudo} show the existence of a two-dimensional process satisfying $\|V_t^{-1/2}S_t\| \sim \log\kappa(V_t)$ almost surely. We first describe this example, and then show how to extend it to higher dimensions. In particular, we will construct a process that attains the same rate as the upper bound presented in our Corollary~\ref{cor:lil_vec}, up to a small, absolute constant. We start by describing the example of \citet{de2007pseudo}.

\begin{example}
\label{eg:2dim}
    Let $(\epsilon_t)_{t \geq 1}$ be a sequence of i.i.d.\ $\calN(0, 1)$ random variables, and let $(\calF_t)_{t \geq 0}$ be the natural filtration associated with $(\epsilon_t)_{t \geq 1}$. First, define the regressors $(U_t)_{t \geq 1}$ by $U_1 = 0$ and $U_{t + 1} := \wb{U}_t + \wb{\epsilon}_t$, where for a sequence $(y_t)_{t \geq 1}$ we define $\wb{y}_t := \frac{1}{t}(y_1 + y_2 + \dots + y_t)$. Then, embed these regressors into $\R^2$ by defining the process $(X_t)_{t \geq 1}$ as $X_t := (1, U_t)^\top$. Clearly, by construction, the process $(X_t)_{t \geq 1}$ is $(\calF_t)_{t \geq 0}$-predictable. 

    With these sequentially constructed regressors, one can construct a martingale $(S_t)_{t \geq 0}$ with respect to $(\calF_t)_{t \geq 0}$ given by $S_t := \sum_{s = 1}^t \epsilon_s X_s$ and a corresponding predictable covariance process $(V_t)_{t \geq 0}$ given by $V_t = \sum_{s = 1}^t X_s X_s^\top$. \citet{de2007pseudo} show that the following hold almost surely:
    \begin{enumerate}
        \item  $\gamma_{\max}(V_t) \sim t \left(1 + \sum_{s = 1}^\infty s^{-1}\epsilon_s\right)$,
        \item $\gamma_{\min}(V_t) \sim \frac{\log(t)}{1 + \sum_{s = 1}^\infty s^{-1}\epsilon_s}$, and
        \item $\|V_t^{-1/2}S_t\| \sim \sqrt{\log(t)}$.
    \end{enumerate}
    Noting that $\log\kappa(V_t) = \log(\gamma_{\max}(V_t)/\gamma_{\min}(V_t)) \sim \log(t)$, we see that we have $\|V_t^{-1/2}S_t\| \sim \sqrt{\log\kappa(V_t)}$ almost surely. Further, it is easily checked that $(S_t, V_t)_{t \geq 0}$ is sub-$\psi_N$, per Definition~\ref{def:psi_vec}. Thus, this example shows that the logarithmic dependence on $\kappa(V_t)$ in Theorem~\ref{thm:vector} cannot, in general, be dropped.
\end{example}

While the above example demonstrates the inevitability of having $\log\kappa(V_t)$ appear in non-asymptotic, self-normalized concentration for vector-valued processes, it does not capture dependence on dimensionality. In the next example, we show that there exists sub-$\psi$ processes $(S_t, V_t)_{t \geq 0}$ such that $\|V_t^{-1/2}S_t\| \sim \sqrt{\frac{d}{2}\log\kappa(V_t)},$ showing our upper bounds are within a multiplicative factor $\sqrt{2}$ of optimal.

\begin{example}
\label{eg:ddim}
    Suppose $d$ is even. Let $(S_t^{(1)})_{t \geq 0}, \dots, (S_t^{(d/2)})_{t \geq 0}$ be i.i.d.\  copies of the process constructed in Example~\ref{eg:2dim}, $(V_t^{(1)})_{t \geq 0}, \dots, (V_t^{(d/2)})_{t \geq 0}$ the corresponding predictable covariance processes, and $(\calF_t)_{t \geq 0}$ the smallest filtration for which $(\epsilon_t^{(1)})_{t \geq 1}, \dots, (\epsilon_t^{(d/2)})_{t \geq 1}$ are adapted, i.e.\ the filtration given by $\calF_t := \calF_t^{(1)} \bigvee \cdots \bigvee \calF_t^{(d/2)}$, where $\calF \bigvee \calG$ denotes the ``join'' of $\sigma$-algebras $\calF, \calG$, i.e.\ the smallest $\sigma$-algebra containing both.

    Define the $\R^d$-valued process $(S_t)_{t \geq 0}$ by $S_t := \left( S_t^{(1)}, \dots, S_t^{(d/2)}\right)$, and the corresponding covariance process $(V_t)_{t \geq 0}$ by
    \[
    V_t := \begin{pmatrix}
        V_t^{(1)} & \mathbf{0} & \cdots & \mathbf{0} \\
        \mathbf{0} & V_t^{(2)} & \cdots & \mathbf{0} \\
        \vdots & & \ddots & \mathbf{0} \\
        \mathbf{0} & \cdots & \cdots & V_t^{(d/2)}
    \end{pmatrix}.
    \]

    Clearly $(S_t)_{t \geq 0}$ is $(\calF_t)_{t \geq 0}$-adapted and $(V_t)_{t \geq 0}$ is $(\calF_t)_{t \geq 0}$-predictable. Moreover, it can readily be checked that $(S_t, V_t)_{t \geq 0}$ is a sub-$\psi_N$ process.

    Since $V_t$ is a block-diagonal matrix, we clearly have $\gamma_{\max}(V_t) = \max_{i \in [d/2]}\gamma_{\max}(V_t^{(i)})$ and $\gamma_{\min}(V_t) = \min_{i \in [d/2]}\gamma_{\min}(V_t^{(i)})$. Thus, using the reasoning on the almost sure behavior on $\gamma_{\max}(V_t^{(i)})$ and $\gamma_{\min}(V_t^{(i)})$ presented in Example~\ref{eg:2dim}, we see that $\log\kappa(V_t) \sim \log(t)$ almost surely. Further, it isn't hard to see that
    \begin{align*}
        \|V_t^{-1/2}S_t\|^2 &= S_t^\top V_t^{-1} S_t \\
        &= (S_t^{(1)})^\top (V_t^{(1)})^{-1}S_t^{(1)} + \cdots + (S_t^{(d/2)})^\top (V_t^{(d/2)})^{-1}S_t^{(d/2)} \sim \frac{d}{2}\log\kappa(V_t).
    \end{align*}
    Thus, we have shown that, up to small constants, the dependence on $\log\kappa(V_t)$ and $d$ in Theorem~\ref{thm:vector} (and thus the corresponding dependence in Corollary~\ref{cor:lil_vec}) is unimprovable.
\end{example}
\section{Applications of Self-Normalized Concentration}
\label{sec:apps}
\subsection{Applications to Online Linear Regression}
\label{subsec:reg}
We now use our self-normalized bounds to construct confidence ellipsoids for slope estimation in online linear regression. In online linear regression, a statistician interacts with an environment over a sequence of rounds. At the beginning of each round, he adaptively (perhaps using observations from previous rounds) selects a point $X_t \in \R^d$, and then observes noisy feedback $Y_t := \langle X_t, \theta^\ast\rangle + \epsilon_t$, where $\epsilon_t$ represents some mean zero noise variable and $\theta^\ast$ is a fixed slope vector. The goal of the statistician is to produce a \textit{confidence sequence} for the unknown slope vector --- that is, a time indexed sequences of sets that all simultaneously contain the unknown parameter with high probability. We formalize the online linear regression model as follows.

\begin{model}[\textbf{Online Linear Regression}]
\label{model:lin_reg}
Let $(\calF_t)_{t \geq 0}$ be a filtration and $\theta^\ast \in \R^d$ a fixed (unknown) slope vector. The online linear regression model is characterized by three processes: (a)  a $(\calF_t)_{t \geq 0}$-predictable $\R^d$-valued sequence $(X_t)_{t \geq 1}$ representing adaptively-chosen covariates, (b) a $(\calF_t)_{t \geq 0}$-adapted scalar-valued processes $(\epsilon_t)_{t \geq 1}$ representing noise, and (c) $(Y_t)_{t \geq 1}$ given as $Y_t = \langle X_t, \theta^\ast\rangle + \epsilon_t$ representing noisy responses. We assume the residual process $S_t := \sum_{s = 1}^t \epsilon_s X_s$ is sub-$\psi$ with (predictable) variance proxy $V_t := \sum_{s = 1}^t X_s X_s^\top$, where $\psi$ is a super-Gaussian CGF-like function.
\end{model}

We consider two estimators. For a fixed regularization parameter $\rho > 0$, we consider the \textbf{least squares with shrinkage}  estimates $\wh{\theta}_t$ and the \textbf{ridge regression} estimates $\wt{\theta}_t$, which are respectively given by
\[
\wh{\theta}_t := (\vX_t^\top \vX_t \lor \rho I_d)^{-1}\vX_t^\top \vY_t, \quad  \text{and} \quad \wt{\theta}_t := (\vX_t^\top \vX_t + \rho I_d)^{-1} \vX_t^\top Y_t, \quad \forall t \geq 1,
\]
where $\vX_t \in \R^{t \times d}$ has $X_1, \dots, X_t$ as its rows and $\vY_t \in \R^d$ is a column vectors with $Y_1, \dots, Y_t$ as its entries. Clearly $\wh{\theta}_t$ reduces to the standard least-squares estimator when $\gamma_{\min}(\vX_t^\top \vX_t) \geq \rho$.


The assumption that $(S_t, V_t)$ is sub-$\psi$ is often mild. For example, it is satisfied (a) if $\log\E_{t - 1}\exp\{\lambda \epsilon_t\} \leq \psi_N(\lambda)$ for all $\lambda \in \R$ (i.e.\ $\epsilon_t$ is conditionally sub-Gaussian), or (b) if $\|X_t\| \leq 1$ for all $t \geq 1$ and $\log\E_{t - 1}\exp\{\pm\lambda \epsilon_t\} \leq \psi(\lambda)$ for some super-Gaussian $\psi$ and all $\lambda \in [0, \lambda_{\max})$. We prove this in Proposition~\ref{prop:residual} in Appendix~\ref{app:apps}. The assumption that $\|X_t\| \leq 1$ for all $t \geq 1$ in the above can be replaced with the assumption that $\|X_t\| \leq R$ for any fixed $R > 0$ by appropriate rescaling. This type of boundedness assumption is regularly made in the mult-armed bandit literature \citep{abbasi2011improved, chowdhury2017kernelized, lattimore2020bandit}, and thus has practical relevance.

We briefly discuss how confidence ellipsoids are constructed in classical least-squares regression. In this setting, one observes a matrix of covariates $\vX \in \R^{t \times d}$ and a response vector $\vY \in \R^t$ given by $\vY = \vX \theta^\ast + \mathbf{\epsilon}$, where $\epsilon \sim \calN(0, \sigma^2 I_d)$. If $\vX^\top \vX$ is full rank, it is well-known~\citep{scheffe1999analysis, keener2010theoretical} that the least-squares estimate for $\theta^\ast$, given by $\wh{\theta} := (\vX^\top \vX)^{-1} \vX^\top \vY$, satisfies
\[
\sigma^{-2}\|(\vX^\top\vX)^{1/2}(\wh{\theta} - \theta^\ast)\|^2 \sim \chi^2_d,
\]
where $\chi^2_d$ denotes the Chi-squared distribution with $d$ degrees of freedom. Letting $x_{d, \delta}$ denote its $\delta$th upper quantile\footnote{that is, $x_{d, \delta} > 0$ is the unique value satisfying $\P(X \geq x_{q, \delta}) = \delta$, where $X \sim \chi^2_q$} it follows that the set
\[
\calC := \{\theta \in \R^d : \sigma^{-2}\|(\vX^\top \vX)^{1/2}(\wh{\theta} - \theta)\|^2 \leq  x_{q, \delta}\}
\]
forms an exact $1 - \delta$ confidence ellipsoid for $\theta^\ast$ centered at $\wh{\theta}$. 

The above confidence ellipsoid fails to be valid when $\vX$ is no longer fixed or when the added noise variables are no longer i.i.d.\ Gaussian, which is the case presented in our heuristic model above. To circumvent this failure of classical statistical machinery, we can leverage our self-normalized bounds for vector-valued processes to construct confidence ellipsoids for $\theta^\ast$ that are valid across all time steps uniformly. We do exactly this in the following theorem.

\begin{theorem}
\label{thm:reg}
    Consider Model~\ref{model:lin_reg}, let $\delta \in (0, 1)$ be arbitrary and set $V_t := \vX_t^\top \vX_t = \sum_{s = 1}^t X_s X_s^\top$. Then, with probability at least $1 - \delta$, simultaneously for all $t \geq 1$, we have
    \[
    \|(V_t \lor \rho I_d)^{1/2}(\wh{\theta}_t - \theta^\ast)\| < \frac{\sqrt{\gamma_{\min}(V_t \lor \rho I_d)}}{1 - \epsilon}\cdot(\psi^\ast)^{-1}\left(\frac{\alpha}{\gamma_{\min}(V_t \lor \rho I_d)}L_{\rho}(V_t)\right) + \sqrt{\rho}\|\theta^\ast\|\mathbbm{1}_{\gamma_{\min}(V_t) < \rho},
    \]
    where the parameters $\alpha, \epsilon, \beta, h$ and the function $L_\rho$ (which partially masks parameter dependence) are as outlined in Theorem~\ref{thm:vector}. Likewise, for the ridge estimates, we have with probability at least $1 - \delta$, 
    \[
    \|(V_t + \rho I_d)^{1/2}(\wt{\theta}_t - \theta^\ast)\| < \frac{\sqrt{\gamma_{\min}(V_t + \rho I_d)}}{1 - \epsilon}\cdot(\psi^\ast)^{-1}\left(\frac{\alpha}{\gamma_{\min}(V_t + \rho I_d)}L_\rho(V_t + \rho I_d)\right) + \sqrt{\rho}\|\theta^\ast\| ,
    \]
\end{theorem}

We further leverage the above theorem in Appendix~\ref{app:var}, in which we construct time-uniform confidence sets for structural parameters in vector autoregressive models.

\paragraph{Comparison with Existing Bounds:}

Many results concerning finite-sample properties of regression estimators are based either in the setting of fixed design \citep{wainwright2019high, agarwal2018model} or in the the case of independent covariates \citep{kuchibhotla2022least, kuchibhotla2022moving}. Moreover, these results are more often than not concerned with bounding the $\ell_2$-error of the estimator, i.e.\ the quantity $\|\wh{\theta}_t - \theta^\ast\|,$ as opposed to the self-normalized quantities we study. 

The main points of comparison for our results have been derived in the online learning/regression literature. We compare our results to those of \citet{abbasi2011improved}. In their work, \citet{abbasi2011improved} construct a confidence sequence for estimating an unknown slope vector $\theta^\ast$ by utilizing self-normalized concentration for sub-Gaussian processes (in particular, leveraging a Gaussian mixture technique that dates back to Example 4.2 in \citet{de2007pseudo}). While subsequent confidence sequences have been derived in the setting of regression with variance estimation \citep{durand2018streaming}, semiparametric regression with bounded confounding \citep{krishnamurthy2018semiparametric}, and ridge regression in reproducing kernel Hilbert spaces \citep{whitehouse2023improved, abbasi2013online}, we focus just on the original contributions of \citet{abbasi2011improved} since all subsequent results exhibit the same rate and hold only in the setting of sub-Gaussian noise.

\begin{fact}[Theorem 2 of \citet{abbasi2011improved}]
\label{fact:mixture_reg}
Let $(\calF_t)_{t \geq 0}$ be a filtration, let $(X_t)_{t \geq 1}$ be an $(\calF_t)_{t \geq 0}$-predictable sequence in $\R^d$, and let $(\epsilon_t)_{t \geq 1}$ be a real-valued $(\calF_t)_{t \geq 1}$-adapted sequence such that conditional on $\calF_{t - 1}$, $\log\E_{t - 1}\exp\left\{\lambda \epsilon_t\right\} \leq \psi_N(\lambda)$ for all $\lambda \in \R$. Then, for any $\rho > 0$ and $\delta \in (0, 1)$, 
\[
\P\left(\exists t \geq 0 : \|(V_t + \rho I_d)^{1/2}(\wt{\theta_t} - \theta^\ast)\| \geq \sqrt{2\log\left(\frac{1}{\delta}\sqrt{\det(I_d + \rho^{-1}V_t)}\right)} + \sqrt{\rho}\|\theta^\ast\|\right) \leq \delta,
\]
where $\wt{\theta}_t$ is the ridge regression estimator outlined in Theorem~\ref{thm:reg} and $V_t := \sum_{s = 1}^t X_s X_s^\top$.

\end{fact}

We compare our results to Fact~\ref{fact:mixture_reg} in the setting $\psi = \psi_N$, as this is the only setting in which the results of \citet{abbasi2011improved} are valid. We first qualitatively compare the above confidence sequence to the ridge regression one presented in Theorem~\ref{thm:reg}. Both bounds suffer the same dependence on the norm of the unknown slope vectors and differ only in the first term. Namely, as noted earlier, $(\psi^\ast_N)^{-1}(u) = \psi_N^{-1}(u) = \sqrt{2u}$, so in this setting our bound reduces to the form 
\begin{align*}
\|(V_t + \rho I_d)^{1/2}(\wt{\theta}_t - \theta^\ast)\| &\leq \sqrt{2\alpha L_\rho(V_t + \rho I_d)} + \sqrt{\rho}\|\theta^\ast\|\\
&= O\left(\sqrt{\log\log\gamma_{\max}(\rho^{-1}V_t + I_d) + d\log\kappa(V_t + \rho I_d)}\right)
\end{align*}
simultaneously for all $t \geq 0$ with probability at least $1 - \delta$. Thus, the same comparison made in Subsection~\ref{subsec:compare:vec} applies in this setting.

A more interesting comparison is between Fact~\ref{fact:mixture_reg} and the least squares bound of Theorem~\ref{thm:reg}. Whereas the above fact provides convergence guarantees for ridge regression estimates, the first part of Theorem~\ref{thm:reg} applies directly to the unregularized, least-squares estimates of the unknown slope vector. In particular, when $\gamma_{\min}(\vX_t^\top \vX_t) \succeq \rho I_d$, the bound does not depend on $\|\theta^\ast\|,$ the norm of the unknown slope vector. This may be desirable in many statistical settings in which either advanced knowledge of such a bound is unavailable or only a loose bound on the quantity is known. Moreover, this bound is interesting in itself as no shift in covariance is required in constructing the confidence ellipsoids. 

\subsection{A Self-Normalized, Multivariate Empirical Bernstein Inequality for Bounded Vectors}
\label{subsec:emp_bern}

We now construct a multivariate empirical Bernstein inequality, extending the Theorem 4 of \citet{howard2021time} to higher dimensions. Empirical Bernstein-style bounds serve as a useful tool in common statistical tasks such as forming confidence sequences
for estimating unknown means \citep{waudby2020estimating}. These bounds are of practical importance as they inherently adapt to the variance of a sequence of observations. If actual observations are tightly clustered, the resulting confidence bounds will be tighter. Likewise, if observations are well-dispersed, the resulting confidence set will be more conservative.  To apply empirical Bernstein these bounds, a statistician must only know that the observations belong to a some bounded set.

To the best of our knowledge, we provide the first multivariate, self-normalized empirical Bernstein. Existing bounds either only hold in the scalar setting \citep{waudby2020estimating, howard2021time}, or do not normalize the quantity being estimated by the accumulated variance process (See, for instance, the work of \citet{cutkosky2019combining} in the case of Hilbert space-valued variables). Providing confidence ellipsoids for mean estimation is desirable as it allows the confidence sets to reflect the ``total amount of information'' gathered in any given direction.

We now present the primary result of this section. In our result, we focus on the case where all observations have norm bounded above by $1/2$ for simplicity. This is mostly for theoretical convenience. While the more general setting where $(X_t)_{t \geq 1}$ belongs to some bounded, convex set is of interest, it can be readily analyzed by reducing to the case where observations lie in $\frac{1}{2}\B_d$\footnote{If $(X_t)_{t \geq 1}$ lies in some arbitrary convex, bounded set $K \subset \R^d$, we can first compute the outer John ellipsoid $E$ of $K$, which is the minimal volume ellipsoid containing the convex set $K$ \citep{john2014extremum}. In many settings, such as in the setting where $K$ belongs to certain families of polytopes, there are computationally efficient algorithms that compute $E$ \citep{cohen2019near, song2022faster}. With $E$ at hand, we can ``recenter'' our observations by defining a new sequence $(X_t')_{t \geq 0}$ by $X_t' := X_t - p$, where $p := \int_{E}xdx$ is the center of mass of $E$. We then have the equality $E - p = \frac{1}{2}A^{1/2}\B_d$, where $A$ is some positive semi-definite matrix. We thus transform our observations into a final sequence $(X_t'')_{t \geq 0}$ defined by $X_t'' := A^{-1/2}(X_t - p)$, which lies almost surely in $\frac{1}{2}\B_d$.}.

For the remainder of this section, we adopt the notation $\ell^\delta_\rho$ and $L^\delta_\rho$ instead of $\ell_\rho$ and $L_\rho$ to explicitly make known the dependence on the confidence parameter $\delta$. We make this dependence explicit as we will be union bounding in the sequel, and thus it will be useful to track the dependence.

\begin{theorem}
\label{thm:emp_bern}
Let $(X_t)_{t \geq 1}$ be a sequence of random vectors in $\R^d$ such that $\|X_t\| \leq 1/2$ almost surely, for all $t \geq 1$, and let $(\calF_t)_{t \geq 0}$ be a filtration to which $(X_t)_{t \geq 1}$ is adapted. Then, the process $(S_t)_{t \geq 0}$ given by $S_t := \sum_{s = 1}^t(X_s - \E_{s - 1}X_s)$ is sub-$\psi_{E, 1}$ with variance proxy $(V_t)_{t \geq 0}$ given by $V_t := \sum_{s = 1}^t (X_s - \wh{\mu}_{s - 1})(X_s - \wh{\mu}_{s  - 1})^\top$, where $\wh{\mu}_t := t^{-1}\sum_{s = 1}^t X_s$. Thus, by Theorem~\ref{thm:vector}, for any fixed choice of parameters $\rho, \alpha, \delta, \beta, \epsilon, h$, we have

\begin{align*}
\P\left(\exists t \geq 0 : \left\|(V_t \lor \rho)^{-1/2}S_t\right\| \geq \frac{\sqrt{\gamma_{\min}(V_t \lor \rho)}}{1 - \epsilon}\cdot(\psi_{E, 1}^\ast)^{-1}\left(\frac{\alpha L_\rho^\delta(V_t)}{\gamma_{\min}(V_t\lor \rho)}\right)\right) \leq \delta.
\end{align*}

In particular, since a sub-$\psi_{E, 1}$ process is sub-$\psi_{G, 1}$, this implies that, with probability at least $1 - \delta$, simultaneously for all $t \geq 0$,
\begin{align*}
\|(V_t \lor \rho)^{-1/2}S_t\| \leq \sqrt{2\alpha L^\delta_\rho(V_t)} + \frac{\alpha L^\delta_\rho(V_t)}{\gamma_{\min}(V_t \lor \rho)}.
\end{align*}

\end{theorem}

We now compare the bound presented in Theorem~\ref{thm:emp_bern} to existing empirical Bernstein-style results. In particular, our main point of comparison will be the following, scalar-valued bound from \citet{howard2021time}.

\begin{prop}[\textbf{\citep[Theorem 4]{howard2021time}}]
\label{prop:emp_bern}
Suppose $(X_t)_{t \geq 1}$ satisfies $X_t \in [-1/2, 1/2]$ almost surely for all $t \geq 1$, and let $(S_t)_{t \geq 0}$, $(\wh{\mu}_t)_{t \geq 0},$ and $(V_t)_{t \geq 0}$ be as in Theorem~\ref{thm:emp_bern}. For any choice of parameters $\alpha, \delta, h, \rho$, we have with probability at least $1 - \delta$, simultaneously for all $t \geq 1$,
\[
|S_t| \leq \sqrt{k_1^2(V_t \lor \rho)\ell^{2\delta}_\rho(V_t) + k_2^2\ell^{2\delta}_\rho(V_t)^2} + k_2\ell^{2\delta}_{\rho}(V_t), 
\]
where $k_1 := \frac{\alpha^{1/4} + \alpha^{-1/4}}{\sqrt{2}},$ $k_2 := \frac{\sqrt{\alpha} + 1}{\sqrt{2}}$, and $\ell_\rho^\delta$ is as given in Theorem~\ref{thm:scalar}.

\end{prop}

Note that $\ell^{2\delta}_\rho$ appears as opposed to $\ell^{\delta}_\rho$ in Proposition~\ref{prop:emp_bern} due to an application of a a union bound in controlling both the upper and lower tail of $S_t$. In the case $d = 1$, Theorem~\ref{thm:emp_bern} yields that, with probability at least $1 - \delta$, $|S_t| \leq \frac{1}{1 - \epsilon}\left[\sqrt{2\alpha(V_t \lor \rho)L^\delta_\rho(V_t)} + L_\rho(V_t)\right]$.  This serves as a sanity check, showing that up to small constants, the univariate bound presented in Theorem~\ref{thm:emp_bern} is equivalent to that in Proposition~\ref{prop:emp_bern}. While one may expect the bound from Proposition~\ref{prop:emp_bern} to be tighter for large values of $V_t$ (as discussed in Section~\ref{subsec:scalar:compare}), this multiplicative gap can be made arbitrarily small by appropriately selecting tuning parameters. 

\section{Proofs of Main Results}
\label{sec:proof}

In this section, we provide the proofs of what we view as the primary two results of this paper: Theorem~\ref{thm:scalar} and Theorem~\ref{thm:vector}. We additionally prove Corollary~\ref{cor:lil_vec}, which, while not a primary contribution of this work, has a proof that is similar in spirit to the other two results derived in this section.  We start with the proof of Theorem~\ref{thm:scalar}, as the scalar bounds derived will play an integral role in the proof of Theorem~\ref{thm:vector}

\begin{proof}[\textbf{Proof of Theorem~\ref{thm:scalar}}]
First, observe that it suffices to show that, in the case $(S_t, V_t)_{t \geq 0}$ is sub-$\psi$ and $V_t \geq 1, \forall t \geq 0$, we have
\begin{equation}
\label{eq:simple_case}
\P\left(\exists t \geq 0: S_t \geq V_t \cdot (\psi^{\ast})^{-1}\left(\frac{\alpha}{V_t}\ell_1(V_t)\right)\right) \leq \delta,
\end{equation}
because, in the general case, we can consider the rescaled process $(S_t', V_t') := (S_t/\sqrt{\rho}, (V_t \lor \rho)/\rho)_{t \geq 0}$ and apply the concentration result from the case where $\rho = 1$. In more detail, clearly by construction $V_t' \geq 1$ for all $t \geq 1$, and by Proposition~\ref{prop:sub_psi}, we know $(S_t', V_t')_{t \geq 0}$ is sub-$\psi_\rho$, where we recall $\psi_\rho(\cdot) = \rho\psi(\cdot/\sqrt{\rho})$. Thus, noting that $(\psi_\rho^\ast)^{-1}(x) = \sqrt{\rho}(\psi^\ast)^{-1}(x/\rho)$ (Proposition~\ref{prop:cgf_rescale}), we have
\begin{align*}
\delta &\geq \P\left(\exists t \geq 0 : S_t' \geq V_t' \cdot(\psi^{\ast}_\rho)^{-1}\left(\frac{\alpha}{V_t'}\ell_1(V_t')\right)\right) \\
&= \P\left(\exists t \geq 0 : \frac{S_t}{\sqrt{\rho}} \geq \frac{V_t \lor \rho}{\rho}\cdot\left(\psi_\rho^{\ast}\right)^{-1}\left(\frac{\alpha\rho}{V_t \lor \rho}\ell_1\left(\frac{V_t \lor \rho}{\rho}\right)\right)\right) \\
&= \P\left(\exists t \geq 0 : \frac{S_t}{\sqrt{\rho}} \geq \frac{V_t \lor \rho}{\rho}\sqrt{\rho}\cdot(\psi^\ast)^{-1}\left(\frac{\alpha}{V_t \lor \rho}\ell_\rho(V_t)\right)\right) \\
&= \P\left(\exists t \geq 0 : S_t \geq (V_t \lor \rho)\cdot \left(\psi^{\ast}\right)^{-1}\left(\frac{\alpha}{V_t \lor \rho}\ell_\rho(V_t)\right)\right),
\end{align*}
which demonstrates the claimed bound in the theorem statement. Thus, going forward, we just prove the bound presented in~\eqref{eq:simple_case}.

For $k \in \N$, define the ``intercept and slope'' pair $(x_k, m_k)$ by
\[
x_k := \alpha^{k}(\psi^\ast)^{-1}\left(\frac{\log(h(k)/\delta)}{\alpha^k}\right), \qquad m_k := \alpha^k, 
\]
and define $g_k : \R_{\geq 0} \rightarrow \R_{\geq 0}$ by
\[
g_k(v) := x_k + \fraks\left(\frac{x_k}{m_k}\right)(v - m_k),
\]
where $\fraks$ is the ``slope transform'' outlined in Appendix~\ref{app:cgf}.
Since we have assumed $\lim_{\lambda \uparrow \lambda_{\max}}\psi'(\lambda) = \infty$, we can apply Lemma~\ref{lem:slope_ineq} to obtain
\[
\P\left(\exists t \geq 0 : S_t \geq g_k(V_t)\right) \leq \exp\left\{-m_k\psi^\ast\left(\frac{x_k}{m_k}\right) \right\} = \frac{\delta}{h(k)}.
\]
Now, since $\fraks(u) \leq u$ (Proposition~\ref{prop:slope_transform}), observe that for $\alpha^k \leq v < \alpha^{k + 1}$, we have
\begin{align*}
\min_{j \in \N}g_j(v) &\leq g_k(v) = x_k + \fraks\left(\frac{x_k}{m_k}\right)(v - m_k) \\
&\leq x_k + \frac{x_k}{m_k}(v - m_k) \numberthis \label{eq:key_diff} = v\frac{x_k}{m_k} \\
&= v \cdot(\psi^{\ast})^{-1}\left(\frac{\log(h(k)/\delta)}{\alpha^k}\right) \\
&\leq v \cdot(\psi^{\ast})^{-1}\left(\frac{\alpha}{v}\log\left(\frac{h(\log_\alpha(v))}{\delta}\right)\right) \\
&= v\cdot(\psi^{\ast})^{-1}\left(\frac{\alpha}{v}\ell_1(v)\right),
\end{align*}
where the third inequality comes from the fact $k \leq \log_\alpha(v),$ $h$ is increasing, and $v \leq \alpha^{k + 1}$.
Now, observe that we have, by a union bound
\begin{align*}
    &\P\left(\exists t \geq 0 : S_t \geq V_t\cdot (\psi^\ast)^{-1}\left(\frac{\alpha\ell_1(V_t)}{V_t}\right)\right) \leq \P\left(\exists t \geq 0: S_t \geq \min_{k \in \N}g_k(V_t)\right)\\
    &\qquad=\P\left(\bigcup_{k \in \N}\left\{\exists t \geq 0 : S_t \geq g_k(V_t)\right\}\right) \\
    &\qquad\leq \sum_{k \in \N}\P(\exists t \geq 0 : S_t \geq g_k(V_t)) \\
    &\qquad \leq \delta\sum_{k \in \N}h(k)^{-1} \leq \delta,
\end{align*}
completing the proof.
\end{proof}

We now go about proving Theorem~\ref{thm:vector}. Before proving the theorem, we state a simple geometric lemma that will be needed in proving our result. In short, the following lemma states that a certain change of variables on $\S^{d - 1}$ does not increase the distance between points of a covering to a significant degree. We prove the following in Appendix~\ref{app:lems}.

\begin{lemma}
\label{lem:cov}
Let $K$ be a proper $\epsilon$-cover of $\S^{d - 1}$, and let $\pi : \S^{d - 1} \rightarrow K$ be a projection mapping onto the cover $K$. Let $T$ be a positive-definite matrix, and let $\kappa := \frac{\gamma_{\max}(T)}{\gamma_{\min}(T)}$ denote its condition number. Let $\pi_T : \S^{d - 1} \rightarrow \S^{d - 1}$ be defined as $\pi_T(\nu) := \frac{T^{1/2}\pi(\omega)}{\|T^{1/2}\pi(\omega)\|}$, where $\omega \in \S^{d - 1}$ is the unique element satisfying 
\begin{equation}
\label{eq:cov}
\nu = \frac{T^{1/2}\omega}{\|T^{1/2}\omega\|}.
\end{equation}
Then, for any $\nu \in \S^{d - 1}$, we have
\[
\|\nu - \pi_T(\nu)\| \leq \sqrt{\kappa}\epsilon.
\]
\end{lemma}

With the above lemma we can now prove the main result of the paper.

\begin{proof}[Proof of Theorem~\ref{thm:vector}]
Observe that if $(S_t, V_t)_{t \geq 0}$ is a sub-$\psi$ process (in the sense of Definition~\ref{def:psi_vec}), then so is $(S_t, V_t \lor \rho I_d),$ so it suffices to assume $V_t \succeq \rho I_d$ going forward.

For $j \in \N$, let $K_j$ be a fixed, minimal proper $\frac{\epsilon}{\beta^j}$-cover of the unit sphere $\S^{d - 1}$, and let $N_j := N(\S^{d - 1}, \epsilon/\beta^j, \|\cdot\|)$. Let $\ell_{\rho}^{(j)} : \R_{\geq 0} \rightarrow \R_{\geq 0}$ be the function $\ell_\rho$ defined in Theorem~\ref{thm:scalar} with $\delta$ set to the value $\delta_j$ defined by
\[
\delta_j := \left(\frac{1 - \beta^{-1}}{\beta^j}\right)\delta/N_j.
\]
That is, $\ell_\rho^{(j)}$ is the function given by
\begin{align*}
\ell_\rho^{(j)}(v) &:= \log\left(h\left(\log_\alpha\left(\frac{v \lor \rho}{\rho}\right)\right)\right) + \log\left(\frac{1}{\delta_j}\right) \\
&= \log\left(h\left(\log_\alpha\left(\frac{v \lor \rho}{\rho}\right)\right)\right) + \log\left(\frac{\beta^j}{\delta(1 - \beta^{-1})}N_j\right).
\end{align*}

Now, since $(S_t, V_t)_{t \geq 0}$ is an $\R^d$-valued sub-$\psi$ process, by Definition~\ref{def:psi_vec}, we know that, for any fixed $\nu \in \S^{d - 1}$, $(\langle \nu, S_t \rangle, \langle \nu, V_t \nu\rangle)_{t \geq 0}$ is sub-$\psi$ in the scalar sense of Definition~\ref{def:psi_scal}. Hence, by applying Theorem~\ref{thm:scalar}, for any fixed $\nu \in \S^{d - 1}$, we have
\begin{equation}
\label{ineq:pt_bound}
\P\left(\exists t \geq 0 : \langle \nu, S_t \rangle  \geq \langle\nu, V_t \nu\rangle \cdot(\psi^\ast)^{-1}\left(\frac{\alpha}{\langle\nu, V_t \nu\rangle}\ell^{(j)}_\rho(\langle \nu, V_t \nu\rangle)\right) \right) \leq \left(\frac{1 - \beta^{-1}}{\beta^j}\right)\delta/N_j.
\end{equation}

Noting that $\langle \nu, V_t \nu \rangle \leq \gamma_{\max}(V_t)$ for all $\nu \in \S^{d - 1}$ and that $(\psi^\ast)^{-1}$ is an increasing function of its argument, we see that~\eqref{ineq:pt_bound} still holds with $\ell_\rho^{(j)}(\langle \nu, V_t \nu\rangle)$ replaced by $\ell_\rho^{(j)}(\gamma_{\max}(V_t))$.

Now, for each $j \in \N$, define the ``bad'' event $B_j$ as
\[
B_j := \left\{\exists t \geq 0, \nu \in K_j :  \langle \nu, S_t \rangle  \geq \langle\nu, V_t \nu\rangle \cdot(\psi^\ast)^{-1}\left(\frac{\alpha}{\langle\nu, V_t \nu\rangle}\ell^{(j)}_\rho(\gamma_{\max}(V_t))\right)\right\}.
\]
A straightforward union bound over the $N_j$ elements of the cover $K_j$ alongside~\eqref{ineq:pt_bound} yields that $ \P(B_j) \leq \frac{1 - \beta^{-1}}{\beta^j}\delta$. 
Defining now the global ``bad'' event $B$ as
\[
B := \left\{\exists j\in \N, \exists \nu \in K_j, \exists t \geq 0 : \langle \nu, S_t \rangle  \geq \langle\nu, V_t \nu\rangle \cdot(\psi^\ast)^{-1}\left(\frac{\alpha}{\langle\nu, V_t \nu\rangle}\ell^{(j)}_\rho(\gamma_{\max}(V_t)\rangle)\right)\right\} = \bigcup_{j \in \N}B_j.
\]
An additional straightforward union bound over indices $j \in \N$ yields
\[
\P(B) = \P\left(\bigcup_{j \in \N}B_j\right) \leq \sum_{j \in \N}\P(B_j) \leq (1 - \beta^{-1})\delta\sum_{j \in \N}\beta^{-j} = \delta.
\]
Now, for $j \in \N$ and $t \geq 0$, let $\pi^{(j)}_t := \pi_{V_t}$ be the projection mapping from $\S^{d - 1}$ onto the finite set $K_j(t) := \left\{V_t^{1/2}\nu/\|V_t^{1/2}\nu\| : \nu \in K_j\right\} \subset \S^{d - 1}$, as in Lemma~\ref{lem:cov}. Note that while $K_j(t)$ is a random subset of the unit sphere (through its dependence on the ``accumulated variance'' operator $V_t$ at time $t$), the underlying $\frac{\epsilon}{\beta^j}$-cover $K_j$ of $\S^{d- 1}$ is fixed. Further, for $j \in \N$ and $t \geq 0$, define the event $E_j(t)$ by
\[
E_j(t) := \left\{\beta^{j} \leq \sqrt{\kappa(V_t)} < \beta^{j + 1}\right\}.
\]
On the event $E_j(t)$, for any $j \in \N$ and $t \geq 0$, we have 
\begin{align*}
    \left\|V_t^{-1/2}S_t\right\| &= \sup_{\omega \in \S^{d - 1}}\left\langle \omega, V_t^{-1/2}S_t\right\rangle = \sup_{\omega \in \S^{d - 1}}\left\{ \left\langle \omega - \pi^{(j + 1)}_t(\omega), V_t^{-1/2}S_t \right\rangle + \left\langle \pi^{(j + 1)}_t(\omega), V_t^{-1/2}S_t \right\rangle\right\} \\
    &\leq \sup_{\omega \in \S^{d - 1}}\left\|\omega - \pi^{(j + 1)}_t(\omega)\right\| \cdot \left\|V_t^{-1/2}S_t\right\| + \sup_{\omega \in K_{j + 1}(t)}\left\langle \omega, V_t^{-1/2}S_t \right\rangle \\
    &\leq \frac{\epsilon}{\beta^{j + 1}}\sqrt{\kappa(V_t)}\left\| V_t^{-1/2}S_t\right\| + \sup_{\nu \in K_{j + 1}}\left\langle \frac{V_t^{1/2}\nu}{\|V_t^{1/2}\nu\|}, V_t^{-1/2}S_t\right\rangle \\
    &\leq \epsilon \left\|V_t^{-1/2}S_t\right\| + \sup_{\nu \in K_{j + 1}}\frac{\langle \nu, S_t\rangle}{\sqrt{\langle\nu, V_t\nu \rangle}}.
\end{align*}
In the above, the first equality comes from the variational representation of the norm $\|\cdot\|$ and the second equality comes from adding and subtracting $\left\langle \pi_t^{(j + 1)}(\omega), V_t^{-1/2}S_t \right\rangle$. Further, the first inequality comes from splitting the supremum and applying Cauchy-Schwarz to the first term, the second inequality comes from applying Lemma~\ref{lem:cov} to $\left\|\omega - \pi_t^{(j + 1)}(\omega)\right\|$ and applying the definition of $K_{j + 1}(t)$, and the final inequality comes from simplifying the second term and from observing that, on the event $E_j(t)$, $\sqrt{\kappa(V_t)} < \beta^{j + 1}$.

Further, observe that, on the event $E_j(t)$, we have the inequality
\begin{align*}
\ell_\rho^{(j + 1)}(\gamma_{\max}(V_t)) &= \log\left(h\left(\log_\alpha\left(\frac{\gamma_{\max}(V_t)\lor \rho}{\rho}\right)\right)\right) + \log\left(\frac{1}{\delta(1 - \beta^{-1})}\right) + \log\left(N_j\beta^{j + 1}\right) \\
&\leq \log\left(h\left(\log_\alpha\left(\frac{\gamma_{\max}(V_t) \lor \rho}{\rho}\right)\right)\right) + \log\left(\frac{1}{\delta(1 - \beta^{-1})}\right) \\
&\qquad+ \log\left(\beta \sqrt{\kappa(V_t)}N_{d - 1}\left(\frac{\epsilon}{\beta\sqrt{\kappa(V_t)}}\right)\right) \\
&= L_\rho(V_t).
\end{align*}

In the above, the inequality follows from observing that $\beta^j \leq \sqrt{\kappa(V_t)}$. From this, rearranging, we see that, for any $j \in \N$ and $t \geq 0$, on the event $E_j(t) \cap B^c$ we have

\begin{align*}
\left\|V_t^{-1/2}S_t\right\| &\leq \frac{1}{1 - \epsilon}\sup_{\nu \in K_{j + 1}}\frac{\langle \nu, S_t \rangle}{\sqrt{\langle \nu, V_t \nu \rangle}} \\
&\leq \frac{1}{1 - \epsilon}\sup_{\nu \in K_{j + 1}}\sqrt{\langle \nu, V_t \nu \rangle}\cdot(\psi^\ast)^{-1}\left(\frac{\alpha}{\langle \nu, V_t \nu \rangle}\ell_\rho^{(j)}(\gamma_{\max}(V_t))\right)\\
&\leq \frac{1}{1 - \epsilon}\sup_{\nu \in K_{j + 1}}\sqrt{\langle \nu, V_t \nu \rangle}\cdot(\psi^\ast)^{-1}\left(\frac{\alpha}{\langle \nu, V_t \nu \rangle}L_\rho(V_t)\right)\\
&\leq \frac{1}{1 - \epsilon}\sup_{\nu \in \S^{d- 1}}\sqrt{\langle \nu, V_t \nu \rangle}\cdot(\psi^\ast)^{-1}\left(\frac{\alpha}{\langle \nu, V_t \nu \rangle}L_\rho(V_t)\right).
\end{align*}

If $(j_t)_{t \in \N}$ is any sequence of natural numbers, and we define $G^{(j_t)} := \bigcap_{t \geq 0}\left\{E_{j_t}(t) \cap B^c\right\}$, it is clear that, on the event $G^{(j_t)}$, the inequality
\begin{equation}
\label{ineq:des_vec_gen}
\left\|V_t^{-1/2}S_t\right\| \leq \frac{1}{1-\epsilon}\sup_{\nu \in \S^{d - 1}}\sqrt{\langle \nu, V_t \nu \rangle}\cdot(\psi^\ast)^{-1}\left(\frac{\alpha}{\langle \nu, V_t \nu \rangle}L_\rho(V_t)\right) 
\end{equation}
holds simultaneously for all $t \geq 0$. Noting that we have the identity $B^c = \biguplus_{(j_t)_{t \in \N}}G^{(j_t)}$ yields that~\eqref{ineq:des_vec_gen} actually holds simultaneously for all $t \geq 0$ on the event $B^c$. What we have done in the above is break the ``good'' event $B^c$ into geometric buckets based on the condition number at each time, and then noted that the regardless of the realized sequence of condition numbers $(\kappa(V_t))_{t \geq 0},$ the target inequality holds.

This proves the claim for arbitrary CGF-like functions $\psi$, which is presented following Theorem~\ref{thm:vector}. Now, if we further assume $\psi$ is super-Gaussian,
on the event $B^c$ defined above, we have 
\begin{align*}
\left\|V_t^{-1/2}S_t\right\| &\leq \frac{1}{1 - \epsilon}\sup_{\nu \in \S^{d - 1}}\sqrt{\langle \nu, V_t \nu \rangle}\cdot(\psi^{\ast})^{-1}\left(\frac{\alpha}{\langle \nu, V_t \nu \rangle}L_\rho(V_t)\right) \\
&= \frac{1}{1 - \epsilon}\sup_{x \in [\gamma_{\min}(V_t), \gamma_{\max}(V_t)]}\sqrt{x}\cdot(\psi^\ast)^{-1}\left(\frac{\alpha}{x}L_\rho(V_t)\right).
\end{align*}
Now, by Lemma~\ref{prop:super_gauss}, we know the assumption that $\psi$ is a super-Gaussian CGF-like function implies that $\psi^{\ast}$ is a sub-Gaussian CGF-like function. Moreover by the same proposition, we see that $\psi^{\ast}(C\cdot)$ is a sub-Gaussian CGF-like function for any positive $C > 0$. Consequently, by Proposition~\ref{prop:super_gauss}, we see that $(\psi^\ast)^{-1}(C u)/\sqrt{u}$ is an increasing function of $u$, and thus by making the change of variable $x := \frac{1}{u}$, that $\sqrt{x}(\psi^{\ast})^{-1}\left(\frac{C}{x}\right)$ a decreasing function of $x$. Thus, we have that, on the event $B^c$ (which, we recall, occurs with probability at least $1 - \delta$)
\[
\left\|V_t^{-1/2}S_t\right\| \leq \frac{1}{1 - \epsilon}\sup_{x \in [\gamma_{\min}(V_t), \gamma_{\max}(V_t)]}\sqrt{x}\cdot(\psi^\ast)^{-1}\left(\frac{\alpha}{x}L_\rho(V_t)\right) \leq \frac{\sqrt{\gamma_{\min}(V_t)}}{1 - \epsilon}\cdot(\psi^\ast)^{-1}\left(\frac{\alpha}{\gamma_{\min}(V_t)}L_\rho(V_t)\right)
\]
simultaneously for all $t \geq 0$, proving the desired result. A symmetric argument holds in the case that the CGF-like function $\psi$ is instead sub-Gaussian, with $\gamma_{\max}(V_t)$ replacing $\gamma_{\min}(V_t)$ in the final inequality.

\end{proof}

Lastly, we prove Corollary~\ref{cor:lil_vec}, which in turn can be used to derive Corollary~\ref{cor:lil_scalar}. While we do not consider this corollary a primary contribution of our work, we include the proof in this section due to its closeness (in spirit) to the previous two proofs.

\begin{proof}[Proof of Corollary~\ref{cor:lil_vec}]
Recalling that $(\psi^\ast_N)^{-1}(u) = \sqrt{2u}$, the assumption that $\psi(\lambda) \sim \frac{\lambda^2}{2}$ implies there, for any $\eta > 0$, there exists an $\overline{u} \in \R_{>0}$ such that
\[
(\psi^\ast)^{-1}(u) \leq (1 + \eta)(\psi_N^\ast)^{-1}(u) = (1 + \eta)\sqrt{2u}
\]
for all $u \in [0, \overline{u}]$. Let $(\alpha_n)_{n \geq 1}, (s_n)_{n \geq 1}, (\epsilon_n)_{n \geq 1},$ and $(\beta_n)_{n \geq 1}$ be such that $\alpha_n, s_n, \beta_n \downarrow 1$ and $\epsilon_n \downarrow 0$ monotonically and let $(\delta_n)_{n \geq 1}$ be such that (a) $\delta_n \downarrow 0$ monotonically and (b) $\sum_{n = 1}^\infty \delta_n < \infty$. Define the sequence of functions $h_n : \N \rightarrow \R_{\geq 0}$ by $h_n(m) := \zeta(s_n)(1 + m)^{s_n}$. Let $T_n$ be the (almost surely finite) random time given by
\begin{align*}
T_n := \inf\Big\{t \geq 0 &: \frac{\alpha}{\gamma_{\min}(V_{t'})}L_1(V_{t'}) \leq \wb{u} \;\;\forall t' \geq t,\;\; \text{and} \\
&\log\left(\frac{C_{d}\zeta(s_n)}{\delta_n(\log(\alpha_n))^{s_n}(1 - \beta_n^{-1})}\right) + d\log\left(\frac{3\beta_n}{\epsilon_n}\right) \leq \frac{s_n}{n}\log\log(\gamma_{\max}(V_t))\Big\}.
\end{align*}
Theorem~\ref{thm:vector} instantiated with the covering number bound in Lemma~\ref{lem:cov_linf} implies that, with probability at least $1 - \delta_n$, simultaneously for all $t \geq T_n$, we have
\begin{align*}
\left\|V_t^{-1/2}S_t\right\| &\leq  \frac{\sqrt{\gamma_{\min}(V_t)}}{1 - \epsilon_n}\cdot(\psi^\ast)^{-1}\left(\frac{\alpha_n}{\gamma_{\min}(V_t)}L_1(V_t)\right) \\
&\leq \frac{1 + \eta}{1 - \epsilon_n}\sqrt{2\alpha_n\left[ s_n \log\log(V_t) + \log\left(\frac{C_d\zeta(s_n)}{\delta_n (\log(\alpha_n))^{s_n}(1 - \beta_n^{-1})}\right) + d\log\left(\frac{3\beta_n\sqrt{\kappa(V_t)}}{\epsilon_n}\right)\right]}\\
&\leq \frac{1 + \eta}{1 - \epsilon_n}\sqrt{2\alpha_n s_n\left(1 + \frac{1}{n}\right)\log\log(V_t) + \alpha_nd\log\kappa(V_t)}.
\end{align*}
Thus, for $n \geq 1$, define the event $A_n$ by
\[
A_n = \left\{ \exists t \geq T_n : \left\|V_t^{-1/2}S_t\right\| \geq \frac{1 + \eta}{1 - \epsilon_n}\sqrt{2\alpha_n s_n\left(1 + \frac{1}{n}\right)\log\log(V_t) + \alpha_nd\log\kappa(V_t)}\right\},
\]
and observe that by the above argument $\P(A_n) \leq \delta_n$. Note that, for arbitrary $\gamma > 1$, we have
\[
A_\gamma := \left\{\|V_t^{-1/2}S_t\| > (1 + \eta)\sqrt{\gamma\left[2\log\log(V_t) + d\log\kappa(V_t)\right]} \;\; \text{i.o.}\right\} \subset \limsup_{n \rightarrow \infty}A_n := \bigcap_{n \geq 1}\bigcup_{k \geq n} A_k,
\]
where i.o.\ denotes an event occurring infinitely often. By the first Borel-Cantelli lemma (see \citet{durrett2019probability}, Chapter 2) we have
\[
\P(A_\gamma) \leq \P\left(\bigcap_{n \geq 1}\bigcup_{k \geq n}A_k\right) = 0,
\]
since $\sum_{n = 1}^\infty \P(A_n) \leq \sum_{n = 1}^\infty \delta_n < \infty$. Thus, with probability 1, we have
\[
\limsup_{t \rightarrow \infty}\frac{S_t}{(1 + \eta)\sqrt{\gamma\left[2\log\log(V_t) + d\log\kappa(V_t)\right]}} \leq 1,
\]
but since $\eta > 0$ and $\gamma > 1$ where arbitrary, the result follows.

\end{proof}

\section{Conclusion and Discussion}

Self-normalized quantities arise naturally in a variety of high-dimensional statistical tasks, with online learning~\citep{abbasi2011online, abbasi2013online, whitehouse2023improved, chowdhury2017kernelized, chugg2023unified}, time series analysis~\citep{bercu2008exponential, shao2015self}, and hypothesis testing~\citep{shekhar2021game, shekhar2022permutation, podkopaev2022sequential, waudby2020estimating} being several notable examples. Despite their crucial role in common statistical tasks, very little has been explored in terms of self-normalized concentration outside of the sub-Gaussian setting.  
In this paper, we present a time-uniform, self-normalized concentration for sub-$\psi$ processes, i.e.\ processes whose increments, roughly, have cumulant generating function bounded by $\psi$. Our results are closed form, have small constants, and have parameters that can be fine-tuned for a statistician's desired application. Moreover, with our bounds, we can establish an asymptotic law of the iterated logarithm for vector-valued processes that recovers the law of iterated logarithm for scalar sub-$\psi$ processes first established by \citet{howard2021time}.

Along with our primary result on the self-normalized concentration of vector-valued processes, we make variety of additional contributions. En route to proving Theorem~\ref{thm:vector}, we prove a non-asymptotic law of the iterated for sub-$\psi$ processes, generalizing the results of \citet{howard2021time} beyond just the sub-Gamma setting. Likewise, we demonstrate how to leverage our self-normalized inequalities in several practical statistical settings. In particular, we derive non-asymptotically valid confidence ellipsoids for online linear regression, describe how to construct confidence sets for vector autoregressive models, and prove a multivariate empirical Bernstein inequality. There are undoubtedly many more settings in which our bounds can be applied, and we leave the exploration of these applications for interesting future work.

While the results presented in this paper are quite general, there are still many interesting questions about self-normalized concentration to be answered. As a first example, existing results on the self-normalized concentration of sub-Gaussian random vectors yield a bound that is proportional to $O\left(\sqrt{\log\det(V_t)}\right)$~\citep{de2007pseudo, victor2009theory, abbasi2011improved}. This is in contrast to the results discussed in this work, which provide bounds of the form $O\left(\sqrt{\log\log\gamma_{\max}(V_t) + d\log\kappa(V_t)}\right)$. As discussed in Section~\ref{sec:vector}, neither form of bound uniformly dominates the other. In particular, when $V_t$ is well-conditioned, our concentration results may be preferable, but for poorly-conditioned $V_t$, determinant rate bounds may be desirable. A major open question is whether determinant rate bounds can be obtained for general sub-$\psi$ processes, or if the determinant rate is just attainable in the sub-Gaussian setting. The techniques discussed in this paper do not seem directly applicable to this setting, and so we thus leave obtaining determinant rate bounds as compelling future work.


This work demonstrates that simple, closed-form bounds on self-normalized processes can be established under very general distributional assumptions. While existing works consider a setting in which the increments of processes are sub-Gaussian, concentration of measure should not be viewed as a ``one size fits all'' phenomenon. For instance, the noise observed in taking real-world may not be sub-Gaussian, but rather perhaps sub-Exponential, sub-Gamma, or even heavy tailed. Overall, our bounds provide a means by which the statistician can properly calibrate confidence in these more delicate settings.
\section{Acknowledgements}

AR acknowledges support from NSF DMS-2310718 and NSF IIS-2229881.
ZSW and JW were supported in part by the NSF CNS2120667, a CyLab 2021 grant, a Google Faculty Research Award, and a Mozilla Research Grant.
JW also acknowledges support from NSF GRFP grants DGE1745016 and DGE2140739.

\bibliography{bib.bib}{}
\bibliographystyle{plainnat}

\appendix
\section{Properties of CGF-like Functions}
\label{app:cgf}

The \textit{cumulant generating function} (or \textit{CGF}) of a random variable plays an integral role in understanding concentration of measure phenomena, such as through the classical Chernoff style of argument~\citep{howard2020time, boucheron2013concentration}. Suppose $X$ is a random variable such that $\E X = 0$, $\E e^{\lambda X} < \infty$ for all $\lambda \in [0, \lambda_{\max})$, and $\lim_{\lambda \uparrow \lambda_{\max}} \E e^{\lambda X} = \infty$. The cumulant generating function of $X$, which can be thought of as ``compressing'' all of the moments of $X$ into a single function, is the mapping $\psi : [0, \lambda_{\max}) \rightarrow \R_{\geq 0}$ given by $\psi(\lambda) := \log\E e^{\lambda X}$. 

In this appendix, we study properties of \textit{cumulant generating function-like} (or \textit{CGF-like}) functions, which are functions that may not be the CGF of any random variable, but display similar analytic properties to CGFs. If $\psi : [0, \lambda_{\max}) \rightarrow \R_{\geq 0}$ is the CGF of a random variable, straightforward calculation yields that $\psi(0) = \psi'(0) = 0$, $\psi''(\lambda) > 0$, and $\psi$ is strictly convex. As such, we say a twice continuously differentiable function $\psi : [0, \lambda_{\max}) \rightarrow \R_{\geq 0}$ is CGF-like if it obeys these aforementioned properties. We study various properties of CGF-like functions in the sequel, as these properties form the foundation of our results studying the self-normalized concentration of sub-$\psi$ processes. 

\begin{prop}
Suppose $\psi : [0, \lambda_{\max}) \rightarrow \R_{\geq 0}$ is CGF-like. Then convex conjugate $\psi^\ast : [0, u_{\max}) \rightarrow \R_{\geq 0}$ defined by $\psi^\ast(u) := \sup_{\lambda \in [0, \lambda_{\max})} \lambda u - \psi(u)$, is also CGF-like, where $u_{\max} := \sup_{\lambda \in [0, \lambda_{\max})} \psi'(\lambda)$.
\end{prop}

\begin{proof}
Clearly $\psi^\ast$ is convex and twice continuously-differentiable. Next, observe that 
\[
\psi^\ast(0) = \sup_{\lambda \in [0, \lambda_{\max})}\left\{ - \psi(\lambda)\right\} = \psi(0) = 0.
\]
Further, using the fact that $(\psi^\ast)' = (\psi')^{-1}$, we have that
\[
(\psi^\ast)'(0) = (\psi')^{-1}(0) = 0.
\]

Lastly, we have that
\[
(\psi^\ast)''(0) = ((\psi')^{-1})'(0) = \frac{1}{\psi''((\psi')^{-1}(0))} = \frac{1}{\psi''(0)} > 0.
\]
Thus, $\psi^\ast$ is also CGF-like.
\end{proof}

If $\psi : [0, \lambda_{\max}) \rightarrow \R_{\geq 0}$ is CGF-like, then for any $\rho > 0$, the ``rescaled'' function $\psi_\rho : [0, \sqrt{\rho}\lambda_{\max}) \rightarrow \R_{\geq 0}$ given by $\psi_\rho(\lambda) := \rho \psi(\lambda/\sqrt{\rho})$ is also CGF-like. These rescaled CGF-like functions arise naturally in studying processes that have been re-normalized to have $V_t \succeq \id_H$ for all $t \geq 0$. These rescaled functions $\psi_\rho$ exhibit the following properties.

\begin{prop}
\label{prop:cgf_rescale}
Let $\psi : [0, \lambda_{\max}) \rightarrow \R_{\geq 0}$ be CGF-like, and let $\psi_\rho : [0, \sqrt{\rho}\lambda_{\max}) \rightarrow \R_{\geq 0}$ be as above. The following hold.
    \begin{enumerate}
        \item $\psi_\rho$ is a CGF-like function.
        \item $\psi_\rho^\ast(u) = \rho\psi^\ast(u/\sqrt{\rho})$.
        \item $(\psi^\ast_\rho)^{-1}(x) = \sqrt{\rho}(\psi^{\ast})^{-1}\left(\frac{x}{\rho}\right)$.
    \end{enumerate} 

\end{prop}

\begin{proof}
The validity of the first claim follows immediately by the definition of a CGF-like function.

To see the validity of the second claim, note that
$$
\psi^\ast_\rho(u) = \sup_{\lambda \in [0, \sqrt{\rho}\lambda_{\max})}\left\{ u \lambda - \rho\psi\left(\frac{\lambda}{\sqrt{\rho}}\right)\right\}.
$$
Differentiating the inner expression on the right-hand side and setting equal to zero furnishes that the supremum is obtained at $\lambda = \sqrt{\rho}(\psi')^{-1}(u/\sqrt{\rho})$. Plugging this back into the above expression yields
\begin{align*}
    \psi^\ast_\rho(u) &= \sqrt{\rho}u (\psi')^{-1}\left(\frac{u}{\sqrt{\rho}}\right) - \rho\psi\left((\psi')^{-1}\left(\frac{u}{\sqrt{\rho}}\right)\right) \\
    &= \rho\left[\frac{u}{\sqrt{\rho}}(\psi')^{-1}\left(\frac{u}{\sqrt{\rho}}\right) - \psi\left((\psi')^{-1}\left(\frac{u}{\sqrt{\rho}}\right)\right)\right]\\
    &= \rho \psi^{\ast}\left(\frac{u}{\sqrt{\rho}}\right),
\end{align*}
which proves the second item.

Lastly, the third item can be readily checked as
\begin{align*}
    \psi_\rho^\ast\left(\sqrt{\rho}(\psi^{\ast})^{-1}\left(\frac{x}{\rho}\right)\right) = \rho (\psi^{\ast})\left(\frac{\sqrt{\rho}}{\sqrt{\rho}}(\psi^\ast)^{-1}\left(\frac{x}{\rho}\right)\right) = \rho \frac{x}{\rho} = x.
\end{align*}
Applying $(\psi_\rho^\ast)^{-1}$ to both sides thus yields the desired result.

\end{proof}

Throughout our work, we are especially interested in studying sub-$\psi$ processes whose increments exhibit tail behavior that is either ``heavier'' or ``lighter'' than that of a Gaussian random-variable. More concretely, we study processes where $\psi$ is a \textit{super-Gaussian} (respectively \textit{sub-Gaussian}) CGF-like function, i.e.\ a CGF-like function where $\frac{\psi(\lambda)}{\lambda^2}$ is a non-decreasing (respectively non-increasing) function of $\lambda$. In words, a CGF-like function $\psi$ is super-Gaussian (or sub-Gaussian) if it increases more rapidly (less rapdily) than the CGF of a standard normal random variable. We focus on super-Gaussian CGF-like functions in the sequel, but exactly analogous results hold for sub-Gaussian CGF-like functions. Super-Gaussian CGF-like functions enjoy a number of convenient properties and equivalent definitions, which we enumerate below.

\begin{prop}
\label{prop:super_gauss}
Suppose $\psi : [0, \lambda_{\max}) \rightarrow \R_{\geq 0}$ is a CGF-like function. The following hold.
\begin{enumerate}
    \item $\psi$ is super-Gaussian if and only if $\psi'(\lambda) \geq \frac{2\psi(\lambda)}{\lambda}$.
    \item If $\psi$ is super-Gaussian, then so is $\varphi := a\psi(b \cdot) : [0, \lambda_{\max}/b) \rightarrow \R_{\geq 0}$ for any $a, b > 0$.
    \item If $\psi$ is super-Gaussian, then $\frac{\psi^{-1}(x)}{\sqrt{x}}$ is a decreasing function of $x \in [0, \infty)$.
    \item $\psi$ is super-Gaussian if and only if its convex conjugate $\psi^\ast$ is sub-Gaussian.
\end{enumerate}

\end{prop}

\begin{proof}

    \begin{enumerate}
        \item Differentiating via the product rule yields
        \[
        \left(\frac{\psi(\lambda)}{\lambda^2}\right)' = \frac{\psi'(\lambda)}{\lambda^2} - \frac{2\psi(\lambda)}{\lambda^3}.
        \]
        Consequently, we have 
        \[
        \left(\frac{\psi(\lambda)}{\lambda^2}\right)' \geq 0 \quad \Leftrightarrow \quad \psi'(\lambda) \geq \frac{2\psi(\lambda)}{\lambda},
        \]
        proving the desired result.
        \item     This result follows from the equivalent condition presented in the first part of the proposition. In particular, observe that we have
        \[
        \varphi'(\lambda) = ab\psi'(b\lambda) \geq 2ab\frac{\psi(b\lambda)}{b\lambda} = \frac{2\varphi(\lambda)}{\lambda},
        \]
        proving the desired result.
        \item  Straightforward calculus yields
        \[
        \left(\frac{\psi^{-1}(x)}{\sqrt{x}}\right)' = \frac{(\psi^{-1})'(x)}{\sqrt{x}} - \frac{1}{2}\frac{\psi^{-1}(x)}{x^{3/2}} = \frac{1}{\sqrt{x}\psi'(\psi^{-1}(x))} - \frac{1}{2}\frac{\psi^{-1}(x)}{x^{3/2}}.
        \]
    
        Next, the assumption of $\psi$ being super-Gaussian yields
        \[
        \psi'(\psi^{-1}(x)) \geq \frac{2\psi(\psi^{-1}(x))}{\psi^{-1}(x)} = \frac{2x}{\psi^{-1}(x)}. 
        \]
        Combining these two panels furnishes
        \[
        \left(\frac{\psi^{-1}(x)}{\sqrt{x}}\right)' = \frac{1}{\sqrt{x}\psi'(\psi^{-1}(x))} - \frac{1}{2}\frac{\psi^{-1}(x)}{x^{3/2}} \leq \frac{1}{2}\frac{\psi^{-1}(x)}{x^{3/2}} - \frac{1}{2}\frac{\psi^{-1}(x)}{x^{3/2}} = 0,
        \]
        which is what we wanted.
        \item We prove the forward direction as the proof of the reverse direction is exactly analogous. Recall that the super-Gaussianity of $\psi$ implies that for all $\lambda \in [0, \lambda_{\max})$, we have $\psi'(\lambda) \geq \frac{2\psi(\lambda)}{\lambda}$. In particular, taking $\lambda = (\psi^\ast)'(u)$ for $u \in [0, u_{\max})$ for $u_{\max} := \sup_{\lambda}\psi'(\lambda)$ yields:
        \[
        u = \psi'((\psi')^{-1}(u)) = \psi'((\psi^\ast)'(u)) \geq \frac{2\psi((\psi^\ast)'(u))}{(\psi^\ast)'(u)}.
        \]
        Rearranging and noting that $\psi = (\psi^\ast)^\ast$ yields
        \begin{align*}
        (\psi^\ast)'(u) &\geq \frac{2\psi((\psi^\ast)'(u))}{u} = \frac{2\sup_{w \in [0, u_{\max})}\left\{w (\psi^\ast)'(u) - \psi^\ast(w)\right\}}{u} \\
        &\geq \frac{2\left\{u (\psi^\ast)'(u) - (\psi^\ast)(u)\right\}}{u} = 2(\psi^\ast)'(u) - \frac{2\psi^\ast(u)}{u}.
        \end{align*}
        Now, subtracting $2(\psi^\ast)'(u)$ from both sides yields
        \[
        -(\psi^\ast)'(u) \geq - \frac{2\psi^\ast(u)}{u}.
        \]
        Multiplying both sides by $-1$ furnishes the desired result.

    \end{enumerate}
\end{proof}

We conclude this section by discussing the slope transform, a recently proposed transform of a CGF-like function that can be used to construct time-uniform, line-crossing inequalities for martingales \citep{howard2020time}. 

\begin{definition}
    Suppose $\psi : [0, \lambda_{\max}) \rightarrow \R_{\geq 0}$ is a CGF-like function. The \textbf{slope transform} associcated with $\psi$ is the mapping $\fraks : [0, u_{\max}) \rightarrow \R_{\geq 0}$ given by
    $$
    \fraks(u) := \frac{\psi\left((\psi^\ast)'(u)\right)}{(\psi^\ast)'(u)}.
    $$
\end{definition}

The slope transform, while abstract and perhaps a bit unintuitive in nature, is of great utility in optimizing our time-uniform, scalar-valued inequalities in the main body of this paper. In particular, we will leverage the following inequality in the proof of Theorem~\ref{thm:scalar}. In the following, recall that for a fixed CGF-like function $\psi : [0, \lambda_{\max}) \rightarrow \R_{\geq 0}$, we defined the quantity $u_{\max}$ as $u_{\max} := \sup_{\lambda}\psi'(\lambda)$. For most examples considered in this paper (in particular in the case of super-Gaussian $\psi$), $u_{\max} = \infty$.

\begin{lemma}[\citet{howard2020time}]
\label{lem:slope_ineq}
Suppose $\psi : [0, \lambda_{\max}) \rightarrow \R_{\geq 0}$ is CGF-like, and suppose $(S_t, V_t)_{t \geq 0}$ is a sub-$\psi$ process, per Definition~\ref{def:psi_scal}. Then, for any $m > 0$, $\delta \in (0, 1)$, and any $x \in (0, m u_{\max})$, we have
\[
\P\left( \exists t \geq 0 : S_t \geq x + \fraks\left(\frac{x}{m}\right)(V_t - m) \right) \leq \exp\left\{-m\psi^\ast\left(\frac{x}{m}\right)\right\}.
\]
\end{lemma}

While the slope transform $\fraks(u)$ may be a generally complicated function, the following upper bound allows us to greatly simplify our analysis. It is proven in \citet{howard2020time}.

\begin{prop}
\label{prop:slope_transform}
    Suppose $\psi : [0, \lambda_{\max}) \rightarrow \R_{\geq 0}$ is CGF-like. Let $\fraks: [0, u_{\max}) \rightarrow \R_{\geq 0}$ be the associated slope transform. Then, for any $u \in [0, u_{\max})$, $\fraks(u) \leq u$.
\end{prop}

\section{Proofs of Results from Section~\ref{sec:apps}}
\label{app:apps}
In this appendix, we provide proofs for all results related to applications of Theorem~\ref{thm:vector}. We start by proving the regression-based results from Subsection~\ref{subsec:reg}, and then move on to proving our empirical Bernstein bound, as discussed in Subsection~\ref{subsec:emp_bern}. We begin by providing practically-relevant examples of when the residual process $S_t = \sum_{s = 1}^t \epsilon_s X_s$ defined in Model~\ref{model:auto} is sub-$\psi$ with variance proxy $V_t = \sum_{s = 1}^t X_s X_s^\top$. 

\begin{prop}
\label{prop:residual}
Suppose $(X_t)_{t \geq 1}$, $(\epsilon_t)_{t \geq 1}$, and $(\calF_t)_{t \geq 0}$ are as outlined in Model~\ref{model:lin_reg}. Let us define the residual process $(S_t)_{t \geq 0}$ by $S_t := \sum_{s = 1}^t \epsilon_s X_s$ and the covariance process $(V_t)_{t \geq 0}$ by $V_t := \sum_{s = 1}^t X_s X_s^\top$. Then, $(S_t)_{t \geq 0}$ is sub-$\psi$ with variance proxy $(V_t)_{t \geq 0}$ if either of the following conditions is satisfied.
\begin{enumerate}
    \item $(\epsilon_t)_{t \geq 1}$ satisfies $\log\E_{t - 1}\exp\left\{\lambda \epsilon_t\right\} \leq \psi_N(\lambda)$ for all $t \geq 1$ and $\lambda \geq 0$.
    \item $\|X_t\| \leq 1$ almost surely for all $t \geq 1$ and $(\epsilon_t)_{t \geq 1}$ satisfies $\log\E_{t - 1}\exp\left\{\lambda \epsilon_t\right\} \leq \psi(\lambda)$ for all $t \geq 1$ and $\lambda \in [0, \lambda_{\max})$, where $\psi : [0, \lambda_{\max}) \rightarrow \R_{\geq 0}$ is some super-Gaussian CGF-like function.
\end{enumerate}

\end{prop}

\begin{proof}
The proof of 1 is straightforward, so we just prove 2. Observe that, from the assumption that $\psi : [0, \lambda_{\max}) \rightarrow \R_{\geq 0}$ is a super-Gaussian CGF-like function, we have that, for any $\lambda_1 < \lambda_2 \in [0, \lambda_{\max})$,
\[
\frac{\psi(\lambda_1)}{\lambda_1^2} \leq \frac{\psi(\lambda_2)}{\lambda_2^2}.
\]

Consequently, for any direction $\nu \in \S^{d - 1}$, $\lambda \in [0, \lambda_{\max})$, and $\|x\| \leq 1$, we have 
\[
\frac{\psi(\lambda\langle \nu, x\rangle )}{\lambda^2\langle \nu, x\rangle^2} \leq \frac{\psi(\lambda)}{\lambda^2}.
\]

Combining this with with the CGF bound on the noise variable $\epsilon_t$ presented in Proposition~\ref{prop:residual} (along with the assumption that $\|X_t\| \leq 1$), we have
\[
\log\E\left(e^{\lambda\langle\nu, X_t\rangle \epsilon_t}\mid \calF_{t - 1}\right) \leq \psi(\lambda\langle \nu, X_t\rangle) \leq \langle \nu, X_t\rangle^2 \psi(\lambda),
\]
where in the above we have used the fact that $X_t$ is $\calF_{t - 1}$-measurable. This immediately yields that, for any $\lambda \in [0, \lambda_{\max})$ and $\nu \in \S^{d -1}$, the process $(M^{\lambda, \nu}_t)_{t \geq 0}$ given by
\[
M_t^{\lambda, \nu} := \exp\left\{\lambda \sum_{s \leq t}\epsilon_s\langle \nu, X_s \rangle - \psi(\lambda)\sum_{s \leq t}\langle \nu, X_s\rangle^2\right\} = \exp\left\{\lambda\langle \nu, S_t \rangle - \psi(\lambda)\langle\nu, V_t \nu\rangle \right\}
\]
is a non-negative supermartingale. Consequently, the scalar-valued process $(\langle \nu, S_t\rangle, \langle\nu, V_t \nu\rangle)_{t \geq 0}$ is sub-$\psi$ for any $\nu \in \S^{d - 1}$. Thus, by definition, the vector process $(S_t, V_t)_{t \geq 0}$ is sub-$\psi$ in the vector-valued sense provided in Definition~\ref{def:psi_vec}. 
    
\end{proof}

We now prove Theorem~\ref{thm:reg}.

\begin{proof}[\textbf{Proof of Theorem~\ref{thm:reg}}]
For a Hermitian matrix $A \in \R^{d \times d}$ let $A \land \rho I_d$ be defined equivalently to $A \lor \rho I_d$ except with the eigenvalue being set to $\gamma_{i}(A) \land \rho$ versus $\gamma_{i}(A) \lor \rho$. Observe that we have the identity 
\begin{equation}
\label{eq:sing_val_decon}
A = A \lor \rho I_d + A \land \rho I_d - \rho I_d.
\end{equation}
Note that we can write the difference between our estimate and the true slope parameter as 
\begin{align*}
\wh{\theta}_t - \theta^\ast &= (V_t \lor \rho I_d)^{-1}\vX_t^\top (\vX_t \theta^\ast + \epsilon_{1:t}) - \theta^\ast \\
&= (V_t \lor \rho I_d)^{-1}(\vX_t^\top \vX_t \lor \rho I_d + \vX_t^\top \vX_t \land \rho I_d - \rho I_d)\theta^\ast + (V_t \lor \rho I_d)^{-1}S_t - \theta^\ast \\
&= (V_t \lor \rho I_d)^{-1}\left(\vX_t^\top \vX_t \land \rho I_d - \rho I_d\right)\theta^\ast + (V_t \lor \rho I_d)^{-1}S_t,
\end{align*}
where in the above we have defined the ``residual process'' $(S_t)_{t \geq 0}$ as $S_t := \sum_{s = 1}^t \epsilon_s X_s \in \R^d$. In the above, the second line follows from the first by applying the equality outlined in Equation~\eqref{eq:sing_val_decon}, the third follows from the second by recalling $V_t = \vX_t^\top \vX_t$ and noting a cancellation between the first and last term.

Thus, applying the triangle inequality gives us
\begin{align*}
\|(V_t \lor \rho I_d)^{1/2}(\wh{\theta}_t - \theta^\ast)\| &\leq \|(V_t \lor \rho I_d)^{-1/2}(\vX_t^\top\vX_t \land \rho I_d - \rho I_d)\theta^\ast\| + \|(V_t \lor \rho I_d)^{-1/2}S_t\| \\
&\leq \sqrt{\rho}\|\theta^\ast\|\mathbbm{1}_{\gamma_{\min}(\vX_t^\top \vX_t) < \rho} + \|(V_t \lor \rho I_d)^{-1/2}S_t\|,
\end{align*}
where the second line follows from the first via straightforward algebraic manipulation and bounding. What remains is to bound $\|(V_t \lor \rho I_d)^{-1/2}S_t\|$. But since we have assumed $(S_t)_{t \geq 0}$ is sub-$\psi$ with variance proxy $(V_t)_{t \geq 0}$, Theorem~\ref{thm:vector} implies that, with probability at least $1 - \delta$, simultaneously for all $t \geq 1$, we have
\[
\|(V_t \lor \rho I_d)^{-1/2}S_t\| \leq \frac{\sqrt{\gamma_{\min}(V_t \lor \rho I_d)}}{1 - \epsilon}\cdot(\psi^\ast)^{-1}\left(\frac{\alpha}{\gamma_{\min}(V_t \lor \rho I_d)}L_\rho(V_t)\right),
\]
which finishes the proof.

Proving the second claim is almost exactly the same. In particular, using a similar line of reasoning, we see that we have the (deterministic) inequality
\begin{align*}
\|(V_t + \rho I_d)^{1/2}(\wt{\theta}_t - \theta^\ast)\| &\leq \rho\|(V_t + \rho I_d)^{-1/2}\theta^\ast\| + \|(V_t + \rho I_d)^{-1/2}S_t\| \\
&\leq \sqrt{\rho}\|\theta^\ast\| + \|(V_t + \rho I_d)^{-1/2}S_t\|,
\end{align*}
where $(S_t)_{t \geq 0}$ is the residual process outlined in the proof of Theorem~\ref{thm:reg}. The result now follows by noting that $(S_t, V_t)_{t \geq 0}$ is sub-$\psi$ in the vector-sense of Definition~\ref{def:psi_vec}.
\end{proof}

Lastly, we prove Theorem~\ref{thm:emp_bern}, which provides a self-normalized, time-uniform empirical Bernstein inequality for multivariate processes. In the proof of Theorem~\ref{thm:emp_bern}, we will need the following lemma, which can be extracted from the proof of Theorem 4 in \citet{howard2021time}, which in turn generalizes a result by~\citet{fan2015exponential}. 

\begin{lemma}[\textbf{Theorem 4 of \citet{howard2021time}}]
\label{lem:emp_ber}
Let $(X_t)_{t \geq 0}$ be a real-valued sequence of random variables adapted to some filtration $(\calF_t)_{t \geq 0}$. Suppose that $|X_t| \leq 1/2$ almost surely for all $t \geq 1$. Then, for any $\lambda \in [0, \lambda)$, the process
\[
L_t^\lambda := \exp\left\{\lambda \sum_{s \leq t}(X_s - \E_{s - 1}X_s) - \psi_{E, 1}(\lambda)\sum_{s \leq t}\left(X_s - \wh{\mu}_{s - 1}\right)^2\right\} 
\]
is a non-negative supermartingale with respect to $(\calF_t)_{t \geq 0}$. Consequently, $\left(\sum_{s = 1}^t(X_s - \E_{s - 1}X_s)\right)_{t \geq 0}$ is sub-$\psi_{E, 1}$ with variance proxy $\left(\sum_{s  = 1}^t(X_s - \wh{\mu}_{s - 1})^2\right)_{t \geq 0}$.
\end{lemma}

We now prove Theorem~\ref{thm:emp_bern}. All we need to do in the proof is check that the process $(S_t, V_t)_{t \geq 0}$ is sub-$\psi_E$ in the sense of Definition~\ref{def:psi_vec}. This boils down to checking that the projection of $(S_t, V_t)_{t \geq 0}$ onto any direction vector is sub-$\psi_E$ in the scalar sense. With Lemma~\ref{lem:emp_ber} in hand, checking this condition becomes trivial. 
\begin{proof}
    To prove the result, it suffices to check that $(S_t, V_t)$ is sub-$\psi_{E, 1}$, per Definition~\ref{def:psi_vec}. Thus, we show that, for any $\nu \in \S^{d - 1}$, $(\langle \nu, S_t\rangle, \langle \nu, V_t \nu \rangle)_{t \geq 0}$ is sub-$\psi_{E, 1}$ in the sense of Definition~\ref{def:psi_scal}. Clearly, $\langle \nu, S_t \rangle \in [-1/2, 1/2]$ almost surely. Further, we have 
    \begin{align*}
    \langle \nu, V_t \nu \rangle &= \sum_{s = 1}^t \langle \nu, (X_s - \wh{\mu}_{s - 1})(X_s - \wh{\mu}_{s - 1})^\top \nu \rangle \\
    &= \sum_{s = 1}^t \langle \nu, X_s - \wh{\mu}_{s - 1}\rangle^2 \\
    &= \sum_{s = 1}^t(\langle \nu, X_s\rangle - \langle \nu, \wh{\mu}_{s - 1}\rangle)^2.
    \end{align*}
    Thus, Lemma~\ref{lem:emp_ber} implies that $(\langle \nu, S_t\rangle, \langle \nu, V_t \nu\rangle)_{t \geq 0}$ is sub-$\psi_{E, 1}$, so the first claim follows. The second claim follows from Theorem~\ref{thm:vector}. Finally, for any $\lambda \in [0, 1)$, 
    \[
    \psi_{E, 1}(\lambda) = -\log(1 - \lambda) - \lambda \leq \frac{\lambda^2}{2(1 - \lambda)} =: \psi_{G, 1}(c),
    \]
    so $(S_t, V_t)_{t \geq 0}$ is sub-$\psi_{G, 1}$ also. Noting that
    \[
    (\psi_{G, 1}^\ast)^{-1}(u) = \sqrt{2u} + u
    \]
    yields the final claim. Proofs of these two facts surrounding $\psi_{E, 1}$ and $\psi_{G, 1}$ can be found in \citet{boucheron2013concentration}.
\end{proof}

\section{Proofs of Technical Lemmas}
\label{app:lems}

In this section, we provides proofs for the technical lemmas used in proving the main results of this paper. We start by proving Lemma~\ref{lem:cov}, which is used in the proof of Theorem~\ref{thm:vector}.

\begin{proof}[Proof of Lemma~\ref{lem:cov}]
    Let $\nu \in \S^{d - 1}$ be arbitrary and let $\omega \in \S^{d - 1}$ be the unique vector satisfying~\eqref{eq:cov}. By the definition of $\omega$ and $\pi_T$, we have
    \begin{align*}
        \|\nu - \pi_T(\nu)\| &= \left\|\frac{T^{1/2}\omega}{\|T^{1/2}\omega\|} - \frac{T^{1/2}\pi(\omega)}{\|T^{1/2}\pi(\omega)\|}\right\| \\
        &\leq \max\left\{ \left\|\frac{T^{1/2}\omega}{\|T^{1/2}\omega\|} - \frac{T^{1/2}\pi(\omega)}{\|T^{1/2}\omega\|}\right\|, \left\| \frac{T^{1/2}\omega}{\|T^{1/2}\pi(\omega)\|} - \frac{T^{1/2}\pi(\omega)}{\|T^{1/2}\pi(\omega)\|}\right\|\right\} \\ 
        &\leq \frac{\gamma_{\max}(T^{1/2})}{\|T^{1/2}\omega\| \land \|T^{1/2}\pi(\omega)\|}\|\omega - \pi(\omega)\| \leq \sqrt{\kappa}\epsilon.
    \end{align*}
    Above, the second inequality follows from pulling out the denominator in each term of the maximum and bounding $\|T^{1/2}(\omega - \pi(\omega))\| \leq \gamma_{\max}(T^{1/2})\|\omega - \pi(\omega)\|$ and the last inequality follows as $\|T^{1/2}\omega\| \land \|T^{1/2}\pi(\omega)\| \geq \gamma_{\min}(T^{1/2})$, and $\|\omega - \pi(\omega)\| \leq \epsilon$ by definition of projection onto a cover. The first inequality follows from a simple calculation. To elaborate, assume $\|T^{1/2}\omega\| \neq \|T^{1/2}\pi(\omega)\|,$ as in the case of equality there is nothing to prove in the inequality. Notice that if $\|T^{1/2} \omega\| < \|T^{1/2}\pi(\omega)\|$, then $q := T^{1/2}\omega/\|T^{1/2}\omega\|$ lies on the surface of the unit ball, and $p := T^{1/2}\pi(\omega)/\|T^{1/2}\omega\|$ lies outside of the unit ball (i.e. has norm greater than 1). The projection of $p$ onto the unit ball is exactly $T^{1/2}\pi(\omega)/\|T^{1/2}\pi(\omega)\|,$ which is closer to $q$ than $p$ since projections onto convex sets decrease Euclidean distance to all points. The maximum above comes from handling the case $\|T^{1/2}\omega\| > \|T^{1/2}\pi(\omega)\|,$ which is analogous.  This shows the desired result.
\end{proof}

We now prove Lemma~\ref{lem:cov_linf}, which is leveraged in the proof of Corollary~\ref{cor:lil_vec}.

\begin{lemma}
\label{lem:cov_linf}
Let $\epsilon \in (0, 1)$ be arbitrary and $d \geq 1$. Then,
\[
N_{d - 1}(\epsilon) \leq C_d\left(\frac{3}{\epsilon}\right)^{d - 1},
\]
where $C_d$ is a constant that does not depend on $\epsilon$.

\end{lemma}

\begin{proof}[Proof of Lemma~\ref{lem:cov_linf}]
We start by providing an upper bound on the proper $\epsilon$-covering number for $\S^{d - 1}_\infty := \{x \in \R^d : \|x\|_\infty := \max_{j \in [d]}|x_j| = 1\}$. Note that we can write
\[
\S^{d - 1}_\infty = \bigcup_{i = 1}^d F_i^+ \cup F_i^-,
\]
where $F_i^+ := \{x \in \R^d : x_i = 1, \|x_{-i}\|_\infty \leq 1\}$ and $F_i^- := \{x \in \R^d : x_i = -1, \|x_{-i}\|_\infty \leq 1\},$ where $x_{-i} \in R^{d - 1}$ is the vector $x$ with the $i$th component omitted. For any $i \in [d]$, $s \in \{+, -\},$ there is a natural isometry between $F_i^s$ and the $(d - 1)$-dimensional $\ell_\infty$ ball defined as $\B_{d - 1}^\infty := \{x \in \R^{d - 1} : \|x\|_\infty \leq 1\}$ given by $x \in \R^d \mapsto x_{-i} \in \R^{d - 1}$. In particular, this implies the proper $\epsilon$-covering number of $F_i^s$ under the $\ell_2$-norm is bounded as
\[
N\left(F_i^s, \epsilon, \|\cdot\|\right) = N\left(\B_{d - 1}^\infty, \epsilon, \|\cdot\|\right) \leq \frac{\mathrm{Vol}_{d - 1}(\B_{d - 1}^\infty)}{\mathrm{Vol}_{d - 1}(\B_{d - 1})}\left(\frac{3}{\epsilon}\right)^{d - 1},
\]
where the last inequality follows from Lemma 5.7 of \citet{wainwright2019high}. From this, we see that we have the bound
\[
N\left(\S^{d - 1}_\infty, \epsilon, \|\cdot\|\right) \leq 2d\frac{\mathrm{Vol}_{d - 1}(\B_{d - 1}^\infty)}{\mathrm{Vol}_{d - 1}(\B_{d - 1})}\left(\frac{3}{\epsilon}\right)^{d - 1} = C_d\left(\frac{3}{\epsilon}\right)^{d - 1},
\]
where we have summed over the $2d$ different $(d - 1)$-dimensional faces $F_1^+, F_1^-, \dots, F_d^+, F_d^-$ and defined the constant $C_d := 2d\frac{\mathrm{Vol}_{d - 1}(\B_{d - 1}^\infty)}{\mathrm{Vol}_{d - 1}(\B_{d - 1})}$.

Let $K$ now denote a minimal proper $\epsilon$-covering of $\S^{d - 1}_\infty$, and let $\pi : \S^{d - 1}_\infty \rightarrow K$ denote the projection onto the covering. Further, let $p : \R^d \rightarrow \B_d$ denote the $\ell_2$ projection onto the unit ball. We claim that the set $K' := \{p(x) : x \in K\}$ is a proper $\epsilon$-covering of $\S^{d - 1}$ under the $\ell_2$ norm. The fact that $p(x) \in \S^{d - 1}$ is immediate. Next, note that for any $y \in \S^{d - 1},$ there is some $x \in \S^{d - 1}_\infty$ such that $p(x) = y$. Observe that $z := p(\pi(x)) \in K'$. Since $p$ is an $\ell_2$ projection, we know that
\[
\|y - z\| = \|p(x) - p(\pi(x))\| \leq \|x - \pi(x)\| \leq \epsilon,
\]
so we have shown that $K'$ is a proper $\epsilon$-covering.
    
\end{proof}

\begin{prop}
\label{prop:sub_psi_heavy}
Suppose $(S_t)_{t \geq 0}$ is a process in $\R^d$ given by $S_t = X_1 + \cdots + X_t$ and $(\calF_t)_{t \geq 0}$ some filtration. If we assume that $\E_t |\langle \nu, X_t\rangle|^3$ is almost surely finite for all $t$ and $\nu$, one can show that $S_t$ is sub-$\psi_{G, 1/6}$ with variance proxy $V_t = \sum_{s = 1}^t\left\{ X_s X_s^\top + \E_{s - 1}\|X_s\|^3 I_d\right\}$.
\end{prop} 

\begin{proof}
For a number $x \in \R$, we let $x_+ := \max\{0, x\}$ and $x_- := \min\{0, x\}$.
Part (h) of Lemma 3 of \citet{howard2020time} yields that, for any $t \geq 0$ and $\nu \in \S^{d - 1}$, the process 
\[
L_t^{\lambda, \nu} := \exp\left\{\lambda \nu^\top S_t - \psi_{G, c}(\lambda)\left([S_t^\nu] + \sum_{s =  1}^t \E_{s - 1}|(\nu^\top X_s)_-|^3\right)\right\}
\]
is a non-negative supermartingale, where $c = 1/6$ and we have let $[S_t^\nu] := \sum_{s = 1}^t \nu^\top X_sX_s^\top\nu$ denote the adapted quadratic variation along direction $\nu$. Observing that 
\[
\sum_{s =1}^t \E_{s - 1}|(\nu^\top X_s)_-|^3 \leq \sum_{s = 1}^t \E_{s - 1}\|X_s\|^3
\]
proves the desired result.
\end{proof}

\section{Applications to Vector Autoregressive Models}
\label{app:var}
We now show how to apply our confidence ellipsoids from Subsection~\ref{subsec:reg} in the section to a vector autoregressive model. We take inspiration from \citet{bercu2008exponential}, who leverage self-normalized concentration results for scalar-valued processes to measure the convergence of least-squares and Yule-Walker estimates for a simple one stage autoregressive model (i.e.\ an $\AR{1}$ model). We focus solely on the least-squares estimates in the sequel. We provide a brief, high-level qualitative comparison between these results and our own. The following results may be of practical interest as autoregressive models and other time series models are frequently applied to problems in econometrics \citep{agarwal2018model, shao2015self} and finance~\citep{pacurar2008autoregressive, darolles2006structural}.

The results we provide in this section are more general than those of \citet{bercu2008exponential} in three ways. First, these authors assume that all noise variables are Gaussian, whereas we allow the noise to be instead conditionally sub-Gaussian. Second, we handle a vector autoregressive model, whereas \citet{bercu2008exponential} only handle the univariate case. Lastly, we handle the problem of general autoregression with $p$-stages of lag, whereas \citet{bercu2008exponential} only handle the case $p = 1$. Our bounds are also different than those of \citet{bercu2008exponential} in that they are derived in terms of the predictable covariance associated with observations, whereas those of \citet{bercu2008exponential} are stated in terms of total number of observations. With these comparisons in hands, we now describe the $p$-stage vector autoregressive model (hereinafter referred to as $\VAR{p}$ for short).

\begin{model}
\label{model:auto}
    A \textbf{$\mathbf{p}$ stage vector-valued autoregressive model}, denoted by $\VAR{p}$, is an $\R^d$-valued process $(Y_t)_{t \geq -p + 1}$ such that $Y_{-p + 1}, \dots, Y_0 \in \R^d$ and $Y_t := \sum_{i = 1}^p A_i Y_{t - i} + \epsilon_t$ where (a) $A_i \in \R^{d \times d}$ are fixed matrices for all $i \in [p]$, and (b) $\epsilon_t$ satisfies $\log\E_{t - 1} \exp\{\lambda \langle \nu, \epsilon_t \rangle\} \leq \frac{\lambda^2}{2}$, where $\nu \in \S^{d - 1}$ and $\lambda \in [0, \lambda_{\max})$. In the above, $\E_{t - 1}[\cdot] := \E\left(\cdot \mid \calF_{t - 1}\right)$, where $(\calF_t)_{t \geq 0}$ is the filtration given by $\calF_t := \sigma(Y_s : -p + 1 \leq  s \leq t)$, for any $t \geq 1$.
\end{model}

For more details on vector autoregressive models, see~\citep{hamilton2020time}. In words, a process $(Y_t)_{t \geq 0}$ satisfies the conditions of a $p$-stage autoregressive (or $\VAR{p}$) model if $Y_t$ is a linear function of $Y_{t - 1}, \dots, Y_{t - p}$ plus mean zero noise. In the above, the values $Y_{-p + 1}, \dots, Y_{0}$ are treated as fixed nonrandom vectors, as is typical in much of the time series analysis literature. However, all results in the sequel still hold if $Y_{-p + 1}, \dots, Y_0$ are random variables that are independent of the noise sequence $(\epsilon_t)_{t \geq 1}$. Typically, the $\VAR{p}$ model also admits a vector mean parameter $\mu \in \R^d$, having the relationship $Y_t = \mu + \sum_{i = 1}^p A_p Y_{t - p} + \epsilon_t$ for all $t \geq 1$, but we omit this to simplify exposition.

The goal of the statistician running an autoregressive model is twofold: (a) to estimate the unknown matrix parameters $A_1, \dots, A_p$, and (b) to calibrate confidence in his estimates. Before discussing classical approaches to estimating these parameters, we simplify notation. We define the ``stacked'' transition matrix $\Pi \in \R^{d \times dp}$ and process vectors $(X_t)_{t \geq 1} \in \R^{dp}$ by
\[
\Pi := \left( A_1, \dots, A_p\right) \quad \text{and} \quad X_t := \left(Y_{t - 1}, Y_{t - 2}, \dots, Y_{t - p}\right)^\top.
\]
For $i \in [d]$, we denote by $\pi(i) \in \R^{dp}$ the $i$th row of the stacked matrix $\Pi$. We likewise denote by $\epsilon_t(i) \in \R$ the $i$th component of the noise vector $\epsilon_t$ and $X_t(i)$ the $i$th component of the state vector $X_t$. Let $\vX_t \in \R^{t \times dp}$ be the matrix with $X_1, \dots, X_t$ as its rows, and let $\vY_t \in \R^{t \times d}$ have $Y_1, \dots, Y_t$ as its rows. For $i \in [d]$, let $\vY_t(i) \in \R^{t}$ denote the $i$th column of $\vY_t$.

If $(\epsilon_t)_{t \geq 1}$ are i.i.d.\ $\calN(0, \sigma^2 I_d)$ with known standard deviation $\sigma$, it is well-known (see \citet{hamilton2020time}, Chapter 11) that the maximum likelihood estimate for $\Pi$ at time $t \geq 1$, for now denoted $\wh{\Pi}_t$, has rows $\wh{\pi}_t(i)$ that are just the least-squares estimates given by
\begin{equation}
\label{eq:row_est}
\wh{\pi}_t(i) := \left(\vX_t^\top\vX_t\right)^{-1}\vX_t^\top \vY_t(i).
\end{equation}
It thus makes sense to study the convergence on these row-wise estimates in the remainder of this section. We focus on studying the convergence of a single row estimate, as the general case follows from union-bounding over the validity of the $d$ row estimates. The proof of the following is a straightforward consequence of Theorem~\ref{thm:reg}, and we provide the brief proof of the result later in this appendix.

\begin{corollary}
\label{cor:auto}
For a fixed coordinate $i \in [d]$, let $(\wh{\pi}_t(i))_{t \geq 1}$ be the sequence of estimates outlined in~\eqref{eq:row_est}.  Let $\rho > 0$ and $\delta \in (0, 1)$ be arbitrary. Define the covariance process $(V_t)_{t \geq 1}$ by $V_t := \vX_t^\top \vX_t$. Then, with probability at least $1 - \delta$, simultaneously for all $t \geq 1$, we have
\[
\|(V_t \lor \rho I_{dp})^{1/2}(\wh{\pi}_t(i) - \pi(i))\| \leq \frac{1}{1 - \epsilon}\sqrt{2\alpha L_\rho(V_t)}+ \sqrt{\rho}\|\theta^\ast\|\mathbbm{1}_{\gamma_{\min}(V_t) < \rho},
\]
where the parameters $\alpha, \epsilon, \beta, h$ and the function $L_\rho$ (which partially masks parameter dependence) are as outlined in Theorem~\ref{thm:vector}.
\end{corollary}

We now compare Corollary~\ref{cor:auto} to traditional asymptotic analyses of equation estimation in the $\VAR{p}$ model. First, note that, in Model~\ref{model:auto} and Corollary~\ref{cor:auto}, we place no assumptions on the matrices $A_1, \dots, A_p \in \R^{d \times d}$. This is in contrast to typical asymptotic analyses, which must assume that all solutions $z \in \C$ to the equation
\begin{equation}
\label{eq:non_exp}
\det(I_d + A_1 z + A_2 z^2 + \dots + A_p z^p) = 0
\end{equation}
have modulus $|z| > 1$ (which we assume holds for validity of the following comparison). In the setting of independent Gaussian noise, as discussed above, the stacked process $(X_t)_{t \geq 1}$ is ergodic and admits some stationary distribution $\pi$ over $\R^{dp}$. It is known that, for any $i \in [d]$, $\sqrt{t}V^{1/2}(\wh{\pi}_t(i) - \pi(i)) \Rightarrow \calN(0, \sigma^2 I_{dp})$, where $V = \E_\pi[X_t X_t^\top] = \lim_{t \rightarrow \infty}\frac{1}{t}\sum_{s = 1}^t X_s X_s^\top$ (the final equality comes from the ergodicity of $(X_t)_{t \geq 1}$). For large $t$, one would thus expect that $\|V^{1/2}_t(\wh{\pi}_t(i) - \pi(i))\| \lesssim \sqrt{dp}$ with high probability. 

We compare our non-asymptotic bounds to this rate. 
Observe that, Corollary~\ref{cor:auto} yields that, with high-probability, simultaneously for all $t \geq 1$,
\[
\|V_t^{1/2}(\wh{\pi}_i(t) - \pi(i))\| = O\left(\sqrt{dp\log\kappa(V_t) + \log\log\left(\gamma_{\max}(V_t)\right)}\right) .
\]
If $\lim_{t \rightarrow \infty}\frac{1}{t}\sum_{s = 1}^t X_s X_s^\top = V$ for some fixed positive-definite matrix $V$ (as will be the case if the $\epsilon_t$ are i.i.d.) and $t$ is a sufficiently large ``target round'', we can view the above as stating $\|V_t^{1/2}(\wh{\pi}_i(t) - \pi(i))\|$ is bounded above by a term growing like $O(\sqrt{\log\log\gamma_{\max}(t) + dp})$ (since $\kappa(V_t) = \kappa(V) = O(1)$ for large $t$, almost surely). As expected in time-uniform concentration, the bounds presented in Corollary~\ref{cor:auto} are looser than those provided by the central limit theorem by a doubly logarithmic factor. 

\paragraph{Comparison with Existing Bounds:}

We lastly make a brief comparison with the bounds of \citet{bercu2008exponential} in the univariate case. In this case, the autoregressive model is parameterized by a scalar $a \in \R$ instead of a sequence of matrices. We thus denote the least-squares estimator of $a$ at time $t \geq 1$ as $\wh{a}_t := \frac{\sum_{s = 1}^tX_{t - 1}X_t}{\sum_{s = 1}^t X_{t - 1}^2}$, departing from our notation of $\wh{\pi}_t$, which was relevant for estimating a row in a stacked matrix. We state the bound of \citet{bercu2008exponential} for convenience.

\begin{fact}[\textbf{Corollary 5.2 of \citet{bercu2008exponential}}]
\label{fact:ar_model}
Suppose $a \in \R$ is fixed. Further, suppose $(Y_t)_{t \geq 0}$ is given by $Y_0 \sim \calN(0, 1)$ and $Y_t := a Y_{t - 1} + \epsilon_t$, where $(\epsilon_t)_{t \geq 1}$ are a sequence of i.i.d.\ $\calN(0, 1)$ random variables independent of $Y_0$. Then, for any fixed $x > 0$ and $t \geq 1$, we have
\[
\P\left(|\wh{a}_{t} - a| \geq x \right) \leq 2\exp\left\{\frac{-tx^2}{2(1 + y_x)}\right\},
\]
where $y_x$ is the unique solution to the equation $\psi_{P, 1}^\ast(y_x) = x^2$, where we recall $\psi_{P, 1}^\ast(u) = (1 + u)\log(1 + u) - u$.
\end{fact}

We draw several high-level comparisons between the bounds. First, the bound in Corollary~\ref{cor:auto} is self-normalized, being defined in terms of the empirical variance $V_t = \sum_{s = 1}^t Y_{s - 1}^2$. The bound in Fact~\ref{fact:ar_model}, on the other hand, depends just on the number of samples used to construct the least-squares estimator, and thus is not self-normalized. Another difference between the conclusions of Fact~\ref{fact:ar_model} and Corollary~\ref{cor:auto} is that Fact~\ref{fact:ar_model} holds only for an individual, fixed sample size $t \geq 1$ whereas Corollary~\ref{cor:auto} is valid for all $t \geq 1$ \textit{simultaneously}. To use Fact~\ref{fact:ar_model} to obtain a time-uniform guarantee, one would need to use a union bound argument to allocate the total failure probability over many rounds. The setting Fact~\ref{fact:ar_model} is also highly parametric, assuming that both the noise and initial state are i.i.d.\ Gaussian random variables. Corollary~\ref{cor:auto}, on the other hand, makes no such assumptions, allowing an arbitrary initial state and conditionally sub-Gaussian noise variables. 

 An explicit comparison of the above bounds is difficult, but we can empirically compare the bounds by simulating a simple $\AR{1}$ model. We provide such a comparison in Figure~\ref{fig:auto_compare} in Appendix~\ref{app:figs}, which plots, for a fixed failure probability $\delta \in (0, 1)$ the autoregressive guarantee from Corollary~\ref{cor:auto} against the corresponding guarantee provided by Fact~\ref{fact:ar_model}. In Subfigure~\ref{fig:auto_compare}\subref{fig:auto_compare:a} we plot the bound from Fact~\ref{fact:ar_model} \textit{without} providing a union bound correction. We thus emphasize that, as plotted, the boundary is only valid point-wise, and not for all $t \geq 1$ simultaneously or for arbitrary stopping times. In Subfigure~\ref{fig:auto_compare}\subref{fig:auto_compare:b}, we make a union bound correction. Figure~\ref{fig:auto_compare} indicates that Corollary~\ref{cor:auto} performs similarly to Fact~\ref{fact:ar_model} when specified to the scalar setting. We believe our bound may be preferable in application over that of Fact~\ref{fact:ar_model} as it is not only significantly more general, but it also inherently adapts to the variance of the observed autoregressive iterates.

We now prove Corollary~\ref{cor:auto}, which concerns the estimation of model parameters in the $\VAR{p}$ model. The proof of the corollary just involves casting the estimation of model parameters in terms of the online linear regression model, i.e.\ Model~\ref{model:lin_reg}. By the assumption that $\psi = \psi_N$, per the discussion following the statement of Theorem~\ref{thm:reg}, it is not necessary to assume $\|X_t\| \leq 1$ for all $t \geq 1$.

\begin{proof}[\textbf{Proof of Corollary~\ref{cor:auto}}]
Let $(\calF_t)_{t \geq 0}$ be the filtration outlined in Model~\ref{model:auto}, i.e. $\calF_t := \sigma(Y_s : -p + 1 \leq s \leq t)$. Note that the $\R^k$-valued sequence $(X_t)_{t \geq 1}$ is $(\calF_t)_{t \geq 1}$-predictable and the $\R^d$-valued noise sequence $(\epsilon_t)_{t \geq 1}$ is $(\calF_t)_{t \geq 0}$-adapted. Further noting the identity
\[
Y_t(i) = \langle \pi(i), X_t\rangle + \epsilon_t,
\]
we see that we are exactly in the setting of Model~\ref{model:lin_reg}. Thus, applying Theorem~\ref{thm:reg} yields the desired result.
\end{proof}

\newpage
\section{Figures}
\label{app:figs}

\begin{figure}[h!]
    \centering
    \subfloat[$\alpha = 1.01$]{
        \includegraphics[width=0.5\textwidth]{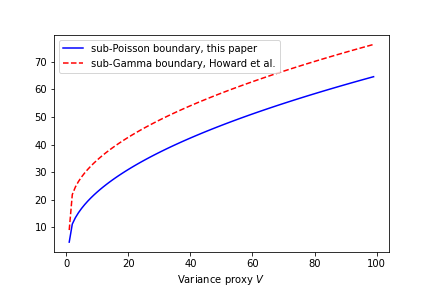}
    }
    \subfloat[$\alpha = 1.05$]{
        \includegraphics[width=0.5\textwidth]{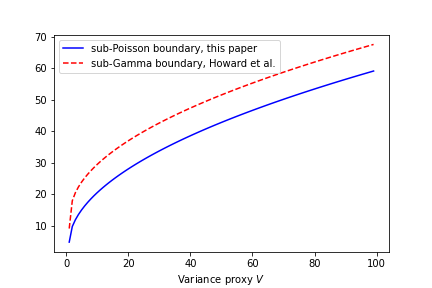}
    }
    \newline
    \subfloat[$\alpha = 1.25$]{
        \includegraphics[width=0.5\textwidth]{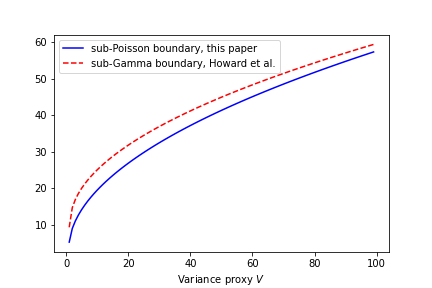}
    }
    \subfloat[$\alpha = 1.5$]{
        \includegraphics[width=0.5\textwidth]{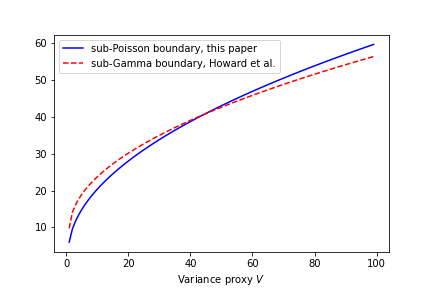}
    }
    \caption{Comparing the boundary of Theorem~\ref{thm:scalar} in the case $\psi = \psi_{P, c}$ with the boundary of Theorem 1 in \citet{howard2021time}, recapped in \eqref{eq:bdry_gamma}. Note that to apply the boundary of \citet{howard2021time}, we need to leverage the fact that a sub-$\psi_{P, c}$ process $(S_t)_{t \geq 0}$ is also sub-$\psi_{G, c}$ with the same variance proxy $(V_t)_{t \geq 0}$. We have made the parameter selection $c = 1$, $\delta = 0.01$, $\rho = 1$, and $h(k) = (1 + k)^2\zeta(2)$, and have correspondingly varied $\alpha$ over several values. We see that for reasonably small choices of intrinsic time spacing $\alpha > 1$, our boundary is tighter than that of \citet{howard2021time}. Thus, we see that although a sub-$\psi_{P, c}$ process can be viewed as a sub-$\psi_{G, c}$ process, this conversion introduces looseness, making our time-uniform concentration result generally preferable in this setting.}
    \label{fig:bdry_poisson}
\end{figure}

\begin{figure}
    \centering
    \subfloat[$\alpha = 1.01$]{
        \includegraphics[width=0.5\textwidth]{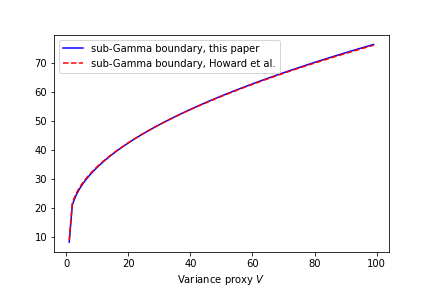}
    }
    \subfloat[$\alpha = 1.05$]{
        \includegraphics[width=0.5\textwidth]{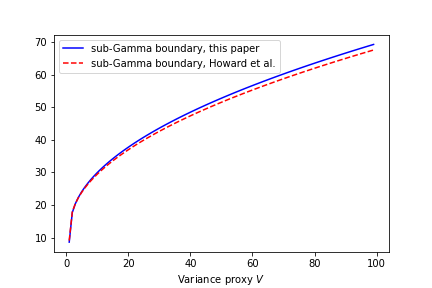}
    }
    \newline
    \subfloat[$\alpha = 1.25$]{
        \includegraphics[width=0.5\textwidth]{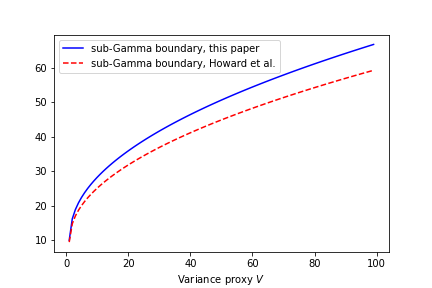}
    }
    \subfloat[$\alpha = 1.5$]{
        \includegraphics[width=0.5\textwidth]{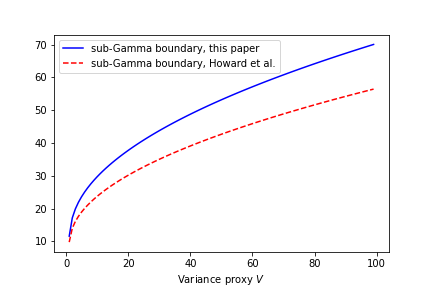}
    }
    \caption{Comparing the boundary of Theorem~\ref{thm:scalar} in the case $\psi = \psi_{G, c}$ with the boundary of \citet{howard2021time} (presented in Equation~\ref{eq:bdry_gamma}). We have made the parameter selection $c = 1$, $\delta = 0.01$, $\rho = 1$, and $h(k) = (1 + k)^2\zeta(2)$, and have correspondingly varied $\alpha$ over several values. As expected from our discussion, our boundary is looser than that of \citet{howard2021time} for all values of $\alpha$, with the gap between the boundaries vanishing as the geometric spacing  $\alpha$ of variance/intrinsic time is decreased towards 1. Since $\alpha = 1.01$ or $\alpha = 1.05$ are reasonable choices for applying our concentration inequalities,  our bounds are just as applicable as those of \citet{howard2021time} even in the sub-Gamma setting.}
    \label{fig:bdry_gamma}
\end{figure}

\begin{figure}
    \centering
    \subfloat[Fact~\ref{fact:ar_model} vs. Corollary~\ref{cor:auto} \textit{without} union bound]{
        \includegraphics[width=0.5\textwidth]{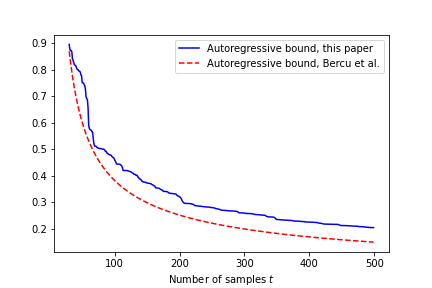}
        \label{fig:auto_compare:a}
    }
    \subfloat[Fact~\ref{fact:ar_model} vs. Corollary~\ref{cor:auto} \textit{with} union bound]{
        \includegraphics[width=0.5\textwidth]{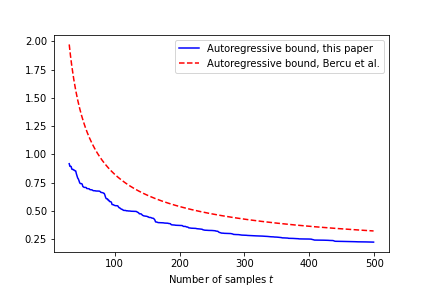}
        \label{fig:auto_compare:b}
    }
    \caption{A comparison of the bounds on $|\wh{a}_t - a|$ provided by Fact~\ref{fact:ar_model} and Corollary~\ref{cor:auto}. In plotting both bounds, we have fixed the failure probability as $\delta = 0.01$. We have numerically solved for $x$ such that the right hand side of Fact~\ref{fact:ar_model} is equal to the target failure probability. When applying Corollary~\ref{cor:auto}, we have set $\alpha = 1.5, h(k) = (1 + k)^2\zeta(2), \rho = 1,$ and note that dependence on $\epsilon$ and $\beta$ can be removed in the univariate case. In Subfigure~\ref{fig:auto_compare}\protect\subref{fig:auto_compare:a}, we plot Fact~\ref{fact:ar_model} point-wise (i.e. we set the failure probability to be $\delta$ for each sample size $t$), and in Subfigure~\ref{fig:auto_compare}\protect\subref{fig:auto_compare:b}, we take a union bound over samples, setting the failure probability to be $\frac{6\delta}{t^2\pi^2}$ for each $t$.}
    \label{fig:auto_compare}
\end{figure}

\end{document}